\documentclass[final,journal]{IEEEtran}
\usepackage[margin=25mm]{geometry}
\usepackage{amssymb,amsmath,amsfonts,amsthm}
\usepackage{graphicx}
\usepackage{url}
\usepackage{showlabels}

\newtheorem{proposition}{Proposition}

\graphicspath{{figs/}}

\author{%
  Mauricio C. de Oliveira%
  \thanks{University of California San Diego, Dept of Mechanical and
    Aerospace Enigineering, La Jolla, CA,
    USA. \texttt{mauricio@ucsd.edu}}%
  
  \today
}
\title{Estimating the Effective Reproduction Number and Variables of
  Disease Models for the COVID-19 Epidemic}

\begin{document}

\maketitle


\begin{abstract}
  This paper deals with the problem of estimating variables in
  nonlinear models for the spread of disease and its application to
  the COVID-19 epidemic. First unconstrained methods are revisited and
  they are shown to correspond to the application of a linear filter
  followed by a nonlinear estimate of the effective reproduction
  number after a change-of-coordinates. Unconstrained methods often
  fail to keep the estimated variables within their physical range and
  can lead to unreliable estimates that require aggressively smoothing
  the raw data. In order to overcome these shortcomings a constrained
  estimation method is proposed that keeps the model variables within
  pre-specified boundaries and can also promote smoothness of the
  estimates. Constrained estimation can be directly applied to raw
  data without the need of pre-smoothing and the associated loss of
  information and additional lag. It can also be easily extended to
  handle additional information, such as the number of infected
  individuals. The resulting problem is cast as a convex quadratic
  optimization problem with linear and convex quadratic
  constraints. It is also shown that both unconstrained and
  constrained methods when applied to death data are independent of
  the fatality rate. The methods are applied to public death data from
  the COVID-19 epidemic.
\end{abstract}

\section{Introduction}

Several authors have attempted to estimate variables and parameters
that can shed light into the progression of the COVID-19
epidemic~\cite{Yang2020,FernandezVillaverde2020,Buckman2020,Anastassopoulou2020}. The
majority of these works utilize classical compartmental epidemic
models~\cite{Hethcote2000}, upon which many predictions and
recommendation regarding COVID-19 are being built
upon~\cite{Ferguson2020}. Such models have also been used to study the
epidemic's behavior in the presence of feedback~\cite{Morato2020,Stewart2020,Cochrane2020}.

Compartmental models are nonlinear low-order continuous-time ordinary
differential equations that are suitable to analysis at population
levels. One of their main features is the relative low complexity and
limited number of parameters which are of easy
interpretation~\cite{Hethcote2000}. Among the existing works that
attempt to estimate such parameters from the available data, for
instance~\cite{Yang2020,FernandezVillaverde2020,Buckman2020,Anastassopoulou2020},
none seem to take advantage of the inherent properties of the model's
variables in the process of estimation. This means that, in the
inevitable presence of noise, the estimated variables and parameters
will often not be compatible with the underlying model and lead to
inconsistent estimates. To mitigate such
difficulties, virtually all works seem to resort to heavy
pre-filtering of the data. Smoothing filters unavoidably
lead to information loss as well as delays in the estimates.

The main contribution of the present paper is to introduce a framework
in which the model variables and parameters can remain constrained
during the process of estimation. One advantage is that all data can
be used for estimation without the need of a smoothing pre-filter,
therefore without incurring the associated data loss and
lag. Susceptible-Infected-Resolving-Deceased-reCovered (SIRDC) models
such as the on in~\cite{FernandezVillaverde2020} are the basic dynamic
models used in the present paper. A number of steps is involved.

First, in Section~\ref{sec:model}, the number of equations in the
model is reduced from five to three. Then a change of coordinates is
introduced with the purpose of isolating all nonlinearities to a
single equation and rewrite the model in terms of the \emph{Effective
  Reproduction Number} ($R$)~\cite{Delamater2019}. Parametrizing the
model in terms of $R$ will be key in rendering certain model
constraints linear.

Based on this reformulated model, in Section~\ref{sec:direct}, the
unconstrained method of~\cite{FernandezVillaverde2020} is shown to be
equivalent to the application of a linear filter followed by the
calculation of a nonlinear estimate for~$R$. This reformulation brings
to light certain properties of the method including the invariance of
the estimate of $R$ on the fatality rate.

In Section~\ref{sec:indirect}, the problem of estimating the variables
and the parameter $R$ of an SRIDC model is reformulated as a quadratic
optimization problem involving an auxiliary linear dynamic system. It
is this reformulation that enables the incorporation of explicit
constraints on the model's variables and the effective reproduction
number, $R$, and its derivative, $\dot{R}$. All such constraints are
shown to be linear in the variables of this auxiliary dynamic
system. The resulting optimization problem is a convex quadratic
program with linear constraints that can be solved efficiently using
off-the-shelf algorithms~\cite{boyd:COP:2004}. Besides enforcing
constraints, the proposed method also allows one to trade-off accuracy
versus smoothness of the estimates, all without requiring any
pre-filtering or smoothing of the data. As with the unconstrained
approach, the constrained estimate of $R$ is also shown to be
independent of the fatality rate. But unlike the unconstrained method,
it can be naturally extended to cover the availability of measurements
of other variables, such as the number of infected individuals.

The paper is closed in Section~\ref{sec:discussion} with brief
conclusion and the application of the constrained estimation method to
publicly available data of the COVID-19 epidemic from several
countries~\cite{Dong2020}.

\section{The SIRDC model for Spread of Disease}
\label{sec:model}

The model used in the present papes is the following
Susceptible-Infected-Resolving-Deceased-reCovered (SIRDC)
model~\cite{FernandezVillaverde2020}. Consider the following
variables:
\begin{itemize}
\item $x_1$: the population Susceptible (S) to a
  disease;
\item $x_2$: the population Infected (I) by a disease;
\item $x_3$: the population Resolving (R) from the
  disease.
\item $x_4$: the population that Died (D) from the
  disease.
\item $x_5$: the population reCovered (C) from the
  disease.
\end{itemize}
In the present paper all variables above are taken as fractions of a
total constant population. The complete SIRDC model is the following
system of nonlinear ordinary differential equations:
\begin{align}
  \label{eq:01a}
  \dot{x}_1 &= - \beta \, x_1 x_2 \\
  \label{eq:01b}
  \dot{x}_2 &= \beta \, x_1 x_2 - \gamma \, x_2 \\
  \label{eq:01c}
  \dot{x}_3 &= \gamma \, x_2 - \theta x_3, \\
  \label{eq:01d}
  \dot{x}_4 &= \delta \theta x_3 \\
  \label{eq:01e}
  \dot{x}_5 &= (1 - \delta) \theta x_3
\end{align}
The main goal of this paper is to estimate all variables in the above
model along with the time-varying parameter~$\beta$. As it will be
seen soon, it is more convenient to work with the \emph{effective
  reproduction number} $R$~\cite{Delamater2019}
\begin{align}
  \label{eq:R}
  R &= \frac{\beta x_1}{\gamma} > 0.
\end{align}
The remaining parameters are characteristic of the disease and here
are assumed to be constant and known:
\begin{itemize}
\item $\gamma$: corresponds to the inverse ammount of time a person is
  infectious, here $5$ days that is $\gamma \approx 0.2$.
\item $\theta$: corresponds to the inverse ammount of time a case
  resolves, here $10$ days that is $\theta \approx 0.1$.
\item $\delta$: the fatality rate, assumed to be
  $0.65$\%~\cite{cdc2020}.
\end{itemize}
Whereas the values of $\gamma$ and $\theta$ are relatively well
studied and can be safely assumed to be known, the fatality rate
$\delta$ carries a large degree of uncertainty, with a variety of
studies producing conflicting numbers and alluding to potential
variations due to local
conditions~\cite{Atkeson2020,Russell2020,Verity2020}. As it will be
seen later, the methods proposed in the present paper, as far as the
estimation of $R$ from death records is concerned, are independent of
the exact knowledge of $\delta$, which will affect the model's
variables but not $R$. The values of $\gamma$ and $\theta$ above were
the ones used in~\cite{FernandezVillaverde2020}.

\subsection{Reduced order model}

Note that not all equations in the SDIRC
model~\eqref{eq:01a}--\eqref{eq:01e} are independent. For instance
\begin{align*}
  \dot{x}_1 + \dot{x}_2 + \dot{x}_3 + \dot{x}_4 + \dot{x}_5 = 0
\end{align*}
which imply
\begin{multline*}
  {x}_1(t) + {x}_2(t) + {x}_3(t) + {x}_4(t) + {x}_5(t) = \\
  {x}_1(0) + {x}_2(0) + {x}_3(0) + {x}_4(0) + {x}_5(0) = 1
\end{multline*}
which reflects the assumption that the total population is
constant. Also
\begin{align*}
  \dot{x}_1 + \dot{x}_2 + \dot{x}_3 + \delta^{-1} \dot{x}_4 = 0
\end{align*}
which implies
\begin{align*}
  {x}_1(t) + {x}_2(t) + {x}_3(t) + \delta^{-1} {x}_4(t) = c,
\end{align*}
where $c$ is a constant. The value of this constant can be determined
as follows. Integrate~\eqref{eq:01d}--\eqref{eq:01e} to obtain
\begin{align*}
  x_4(t) &= x_4(0) + \delta \int_{0}^t \theta x_3(\tau) \, d\tau, \\
  x_5(t) &= x_5(0) + (1 - \delta) \int_{0}^t \theta x_3(\tau) \, d\tau
\end{align*}
from which it follows that
\begin{align*}
  (1 - \delta) \, (x_4(t) - x_4(0)) &=  \delta \, (x_5(t) - x_5(0)).
\end{align*}
For instance, if one assumes that $x_4(0) = x_5(0) = 0$, as in the
beginning of the disease, then
\begin{align*}
\delta^{-1} x_4(t) = x_4(t) + x_5(t)
\end{align*}
which implies that $c = 1$.

The above relationships means that the SIRDC model can be reduced to
its first three equations
\begin{align}
  \label{eq:02a}
  \dot{x}_1 &= - \beta \, x_1 x_2, \\
  \label{eq:02b}
  \dot{x}_2 &= \beta \, x_1 x_2 - \gamma \, x_2, \\
  \label{eq:02c}
  \dot{x}_3 &= \gamma \, x_2 - \theta \, x_3,
\end{align}
since the remaining variables
\begin{align*}
  x_4 &= \delta \left [ 1 - (x_1 + x_2 + x_3) \right ], \\
  x_5 &= (1 - \delta) \left [ 1 - (x_1 + x_2 + x_3) \right ],
\end{align*}
can be obtained from the reduced order model variables. The
measurement
\begin{align}
  \label{eq:02d}
  y &= x_4 = \delta \left [ 1 - (x_1 + x_2 + x_3) \right ]
\end{align}
can also be obtained from the first three variables.

The following basic properties of the variables in the SIRDC model
will be explicitly used later. Because $\beta > 0$, it follows that
\begin{align}
  \label{eq:props}
  x_i &\in [0, 1], \quad i = 1, 2, 3, &
  \dot{x}_1 &\leq 0
\end{align}
so that $x_1$ is monotonically decreasing. Indeed a distinctive aspect
of the approach in the present paper is that such constraints will be
enforced throughout the estimation process.

\subsection{Change of coordinates}

The reduced model~\eqref{eq:02a}--\eqref{eq:02d} is still nonlinear,
with the first and second equations containing products of the model
variables. The following change of coordinates can confine the
nonlinearities to a single equation, a key fact that will be used
afterwards, and express the dynamics in terms of the effective
reproduction number, $R$, defined in~\eqref{eq:R}. Consider the change
of coordinates
\begin{align*}
  \begin{pmatrix}
    z_1 \\ z_2 \\ z_3 \\ R
  \end{pmatrix} &= 
  \begin{pmatrix}
    x_1 \\ x_1 + x_2 \\ x_1 + x_2 + x_3 \\ \beta z_1 / \gamma 
  \end{pmatrix}, &
  \begin{pmatrix}
    x_1 \\ x_2 \\ x_3 \\ \beta
  \end{pmatrix} &= 
  \begin{pmatrix}
    z_1 \\ z_2 - z_1 \\ z_3 - z_2 \\ \gamma R / z_1
  \end{pmatrix}
\end{align*}
and apply it to~\eqref{eq:02a}--\eqref{eq:02d} to obtain the
equivalent model
\begin{align}
  \label{eq:03a}
  \dot{z}_1 &= - \gamma R \, (z_2 - z_1), \\
  \label{eq:03b}
  \dot{z}_2 &= - \gamma \, (z_2 - z_1), \\
  \label{eq:03c}
  \dot{z}_3 &= - \theta \, (z_3 - z_2),
\end{align}
and the measurement
\begin{align}
  \label{eq:03d}
  y &= \delta \left [ 1 - z_3 \right ].
\end{align}
Note that the above change of coordinates is well defined because for
any initial condition in which $x_1(0) > 0$ then $x_1(t) > 0$ for all
$t \leq 0$. The properties~\eqref{eq:props} can be translated in terms
of the new variables as
\begin{align}
  \label{eq:propscov}
  z_i &\in [0, 1], \quad i = 1, 2, 3, &
  z_1 &\leq z_2 \leq z_3, &
  \dot{z}_1 &\leq 0
\end{align}
where the ranking of the new variables come from the fact that they
are accumulated sums of non-negative values.

\subsection{Discrete-time model}
\label{sec:discrete}

For most of the remaining of this paper the following first-order
(Euler) approximation of the system~\eqref{eq:03a}--\eqref{eq:03d}
\begin{align}
  \label{eq:04a}
  z_1(k+1) &= z_1(k) - \gamma R(k) (z_2(k) - z_1(k)), \\
  \label{eq:04b}
  z_2(k+1) &= z_2(k) - \gamma \, ( z_2(k) - z_1(k) ), \\
  \label{eq:04c}
  z_3(k+1) &= z_3(k) - \theta  \, ( z_3(k) - z_2(k) ),
\end{align}
and the measurement
\begin{align}
  \label{eq:04d}
  y(k) &= \delta - \delta \, z_3(k),
\end{align}
will be used. The main goal is to estimate the time-varying parameter
$R(k)$ and the variables $x_1(k)$ through $x_3(k)$ from the
measurement $y(k)$ satisfying the constraints
\begin{multline}
  \label{eq:10}
  {z}_i(k) \in [0, 1], \quad i = 1, 2, 3,  \\
  {z}_1(k) \leq {z}_2(k) \leq {z}_3(k), \quad
  {z}_1(k+1) \leq {z}_1(k),
\end{multline}
which are the discrete-time counterparts to~\eqref{eq:propscov}.

\section{Unconstrained Estimation}
\label{sec:direct}

In the presence of the entire state evolution one could estimate
$R(k)$ by solving equation~\eqref{eq:04a}, that is
\begin{align}
  \label{eq:05}
  \hat{R}(k) &= \frac{z_1(k) - z_1(k+1)}{\gamma (z_2(k) - z_1(k))}.
\end{align}
In fact, assuming that the measurement $y(k)$ is free of noise, it is
possible to recursively rewrite $z_1(k)$, $z_2(k)$ and $z_3(k)$ in
terms of $y(k)$ and apply~\eqref{eq:05}. This process is equivalent to
the method proposed in~\cite{FernandezVillaverde2020}, leading to the
exact same results.

As a first contribution of the present paper, it is shown in the
Appendix, that this recursion amounts to applying the following linear
filter with state space realization
\begin{align}
  \label{eq:008a}
  \hat{x}(k+1) &= A \, \hat{x}(k) + B \, \hat{u}(k) \\
  \label{eq:008b}
  \hat{z}(k) &= C \, \hat{x}(k) + D \, \hat{u}(k)
\end{align}
in which
\begin{align}
  \label{eq:008c}
  A &=
  \begin{bmatrix}
    0 & 0 \\
    1 & 0
  \end{bmatrix}, \quad
  B =
  \begin{bmatrix}
    1 \\
    0
  \end{bmatrix}, \quad
  D =
  \begin{bmatrix}
    \gamma^{-1} \theta^{-1} \\
    0
  \end{bmatrix}, \\
  \label{eq:008d}
  C &= 
  \begin{bmatrix}
    \theta^{-1} \!+ \!\gamma^{-1} \!-\! 2 \, \theta^{-1} \gamma^{-1} & ( 1 \!-\! \theta^{-1})( 1 \!-\! \gamma^{-1}) \\
    \theta^{-1} & (1 - \theta^{-1})
  \end{bmatrix}, &
\end{align}
to the input
\begin{align}
  \label{eq:008e}
  \hat{u}(k) &= 1 - \delta^{-1} y(k),
\end{align}
so as to produce the vector of estimates
\begin{align*}
  \hat{z}(k) &=
  \begin{pmatrix}
    \hat{z}_1(k) \\
    \hat{z}_2(k)
  \end{pmatrix}
\end{align*}
from which $\hat{R}$ can be calculated as in~\eqref{eq:05} after
substituting $z_1$ and $z_2$ by the estimates $\hat{z}_1$ and
$\hat{z}_2$. The filter's initial condition should be initialized as
in
\begin{equation}
  \label{eq:008f}
\begin{aligned}
  \hat{x}(0) &= C^{-1} \left ( \hat{z}(0) - D \hat{u}(0) \right ), \\
  \hat{z}(0) &= 
  \begin{pmatrix}
    \hat{z}_1(0) \\
    \hat{z}_2(0) 
  \end{pmatrix} =
  \begin{pmatrix}
    \hat{x}_1(0) \\
    \hat{x}_1(0) + \hat{x}_2(0)
  \end{pmatrix}
\end{aligned}
\end{equation}
where $\hat{x}_1(0)$ and $\hat{x}_2(0)$ are estimates of the initial
susceptible and infected populations.

\begin{figure}
  \includegraphics[width=\columnwidth]{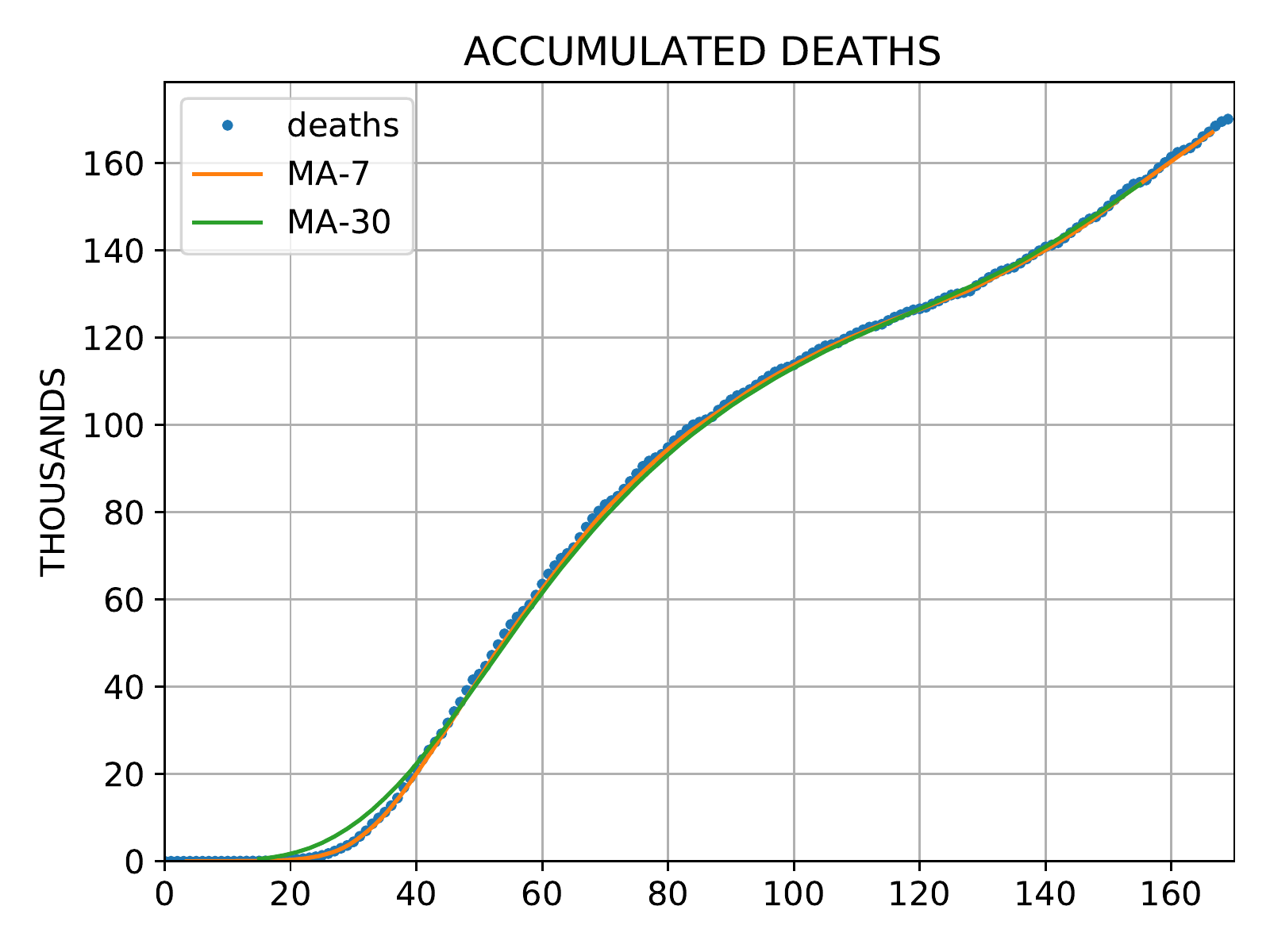}
  \includegraphics[width=\columnwidth]{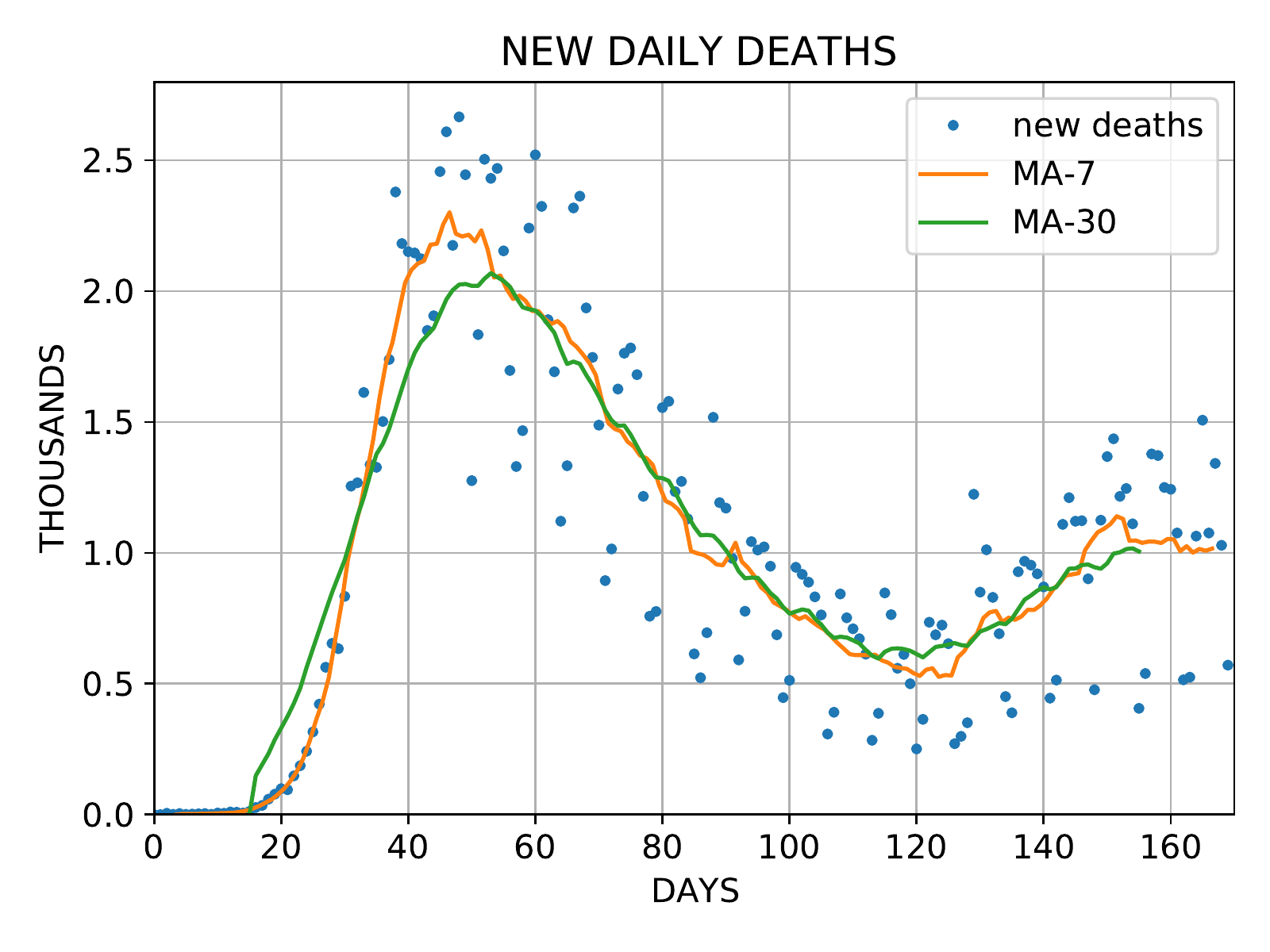}
  \caption{\label{fig:1}Accumulated and new daily deaths attributed to
    COVID--19 in the United States from 01/22/2020 through
    08/16/2020~\cite{Dong2020}. Day $0$ corresponds to the first day in
    this range in which a non-zero number of deaths was reported. Also
    shown is the result of the centered moving average with periods of
    $7$ and $30$ days. Note the large lags introduced by smoothing.}
\end{figure}

\begin{figure}
  \centering
  \includegraphics[width=\columnwidth]{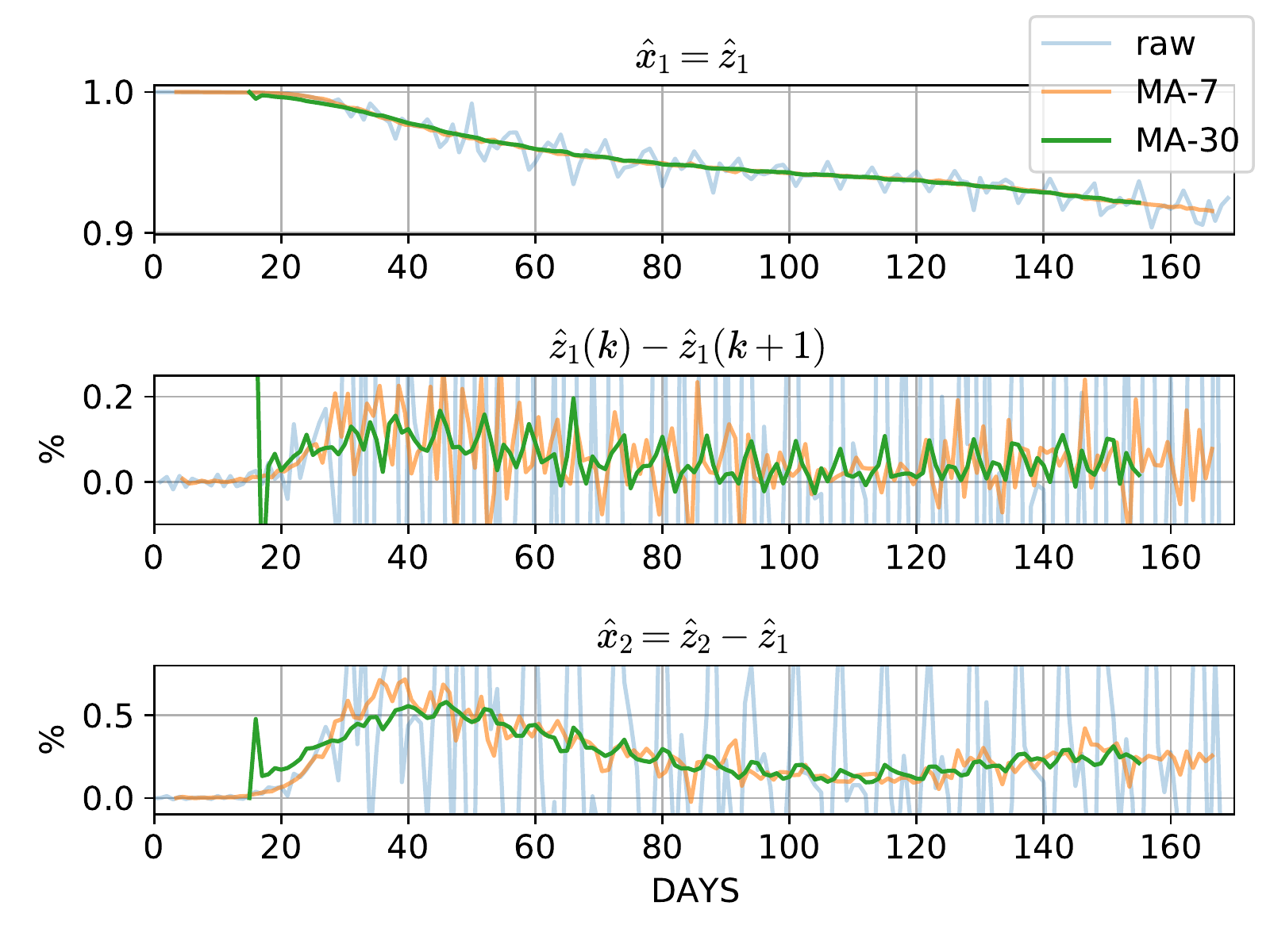}
  \caption{\label{fig:4}Estimates $\hat{x}_1 = \hat{z}_1$ and
    $\hat{x}_2 = \hat{z}_2 - \hat{z}_1$ produced by the
    filter~\eqref{eq:05}--\eqref{eq:008f} for United States data from
    Fig.~\ref{fig:1} for raw and centered moving averages with periods
    $7$ and $30$ days. Note how $\hat{x}_1 = \hat{z}_1$ is not
    monotically decreasing, and $\hat{x}_2 = \hat{z}_2 - \hat{z}_1$ and
    $\hat{z}_1(k) - \hat{z}_1(k+1)$ often become negative, even after
    smoothing.}
\end{figure}

One problem with the above approach is that, in the presence of noise,
there are no guarantees that the estimated variables $\hat{z}_1(k)$
and $\hat{z}_2(k)$ satisfy the structural
constraints~\eqref{eq:10}. Take for example the series of accumulated
COVID-19 deaths in the United States shown in Fig.~\ref{fig:1}
obtained from the COVID-19 Data Repository by the Center for Systems
Science and Engineering (CSSE) at Johns Hopkins
University~\cite{Dong2020}. The corresponding estimates $\hat{z}_1$
and $\hat{z}_2$ and the key differences $\hat{z}_1(k) -
\hat{z}_1(k+1)$ and $\hat{z}_2(k) - \hat{z}_1(k)$ are shown in
Fig.~\ref{fig:4}. Note how these estimates do not
satisfy~\eqref{eq:10}. What this means is that the resulting $\hat{R}$
estimate~\eqref{eq:05} will experience wild swings and even take
negative or very large values. In the example in Fig.~\ref{fig:2},
even the highly smoothed $30$-day moving average estimate swings
below~$0$ and above~$5$. The raw and $7$-day moving average estimates
are basically useless.

\begin{figure}
  \centering
  \includegraphics[width=.5\textwidth]{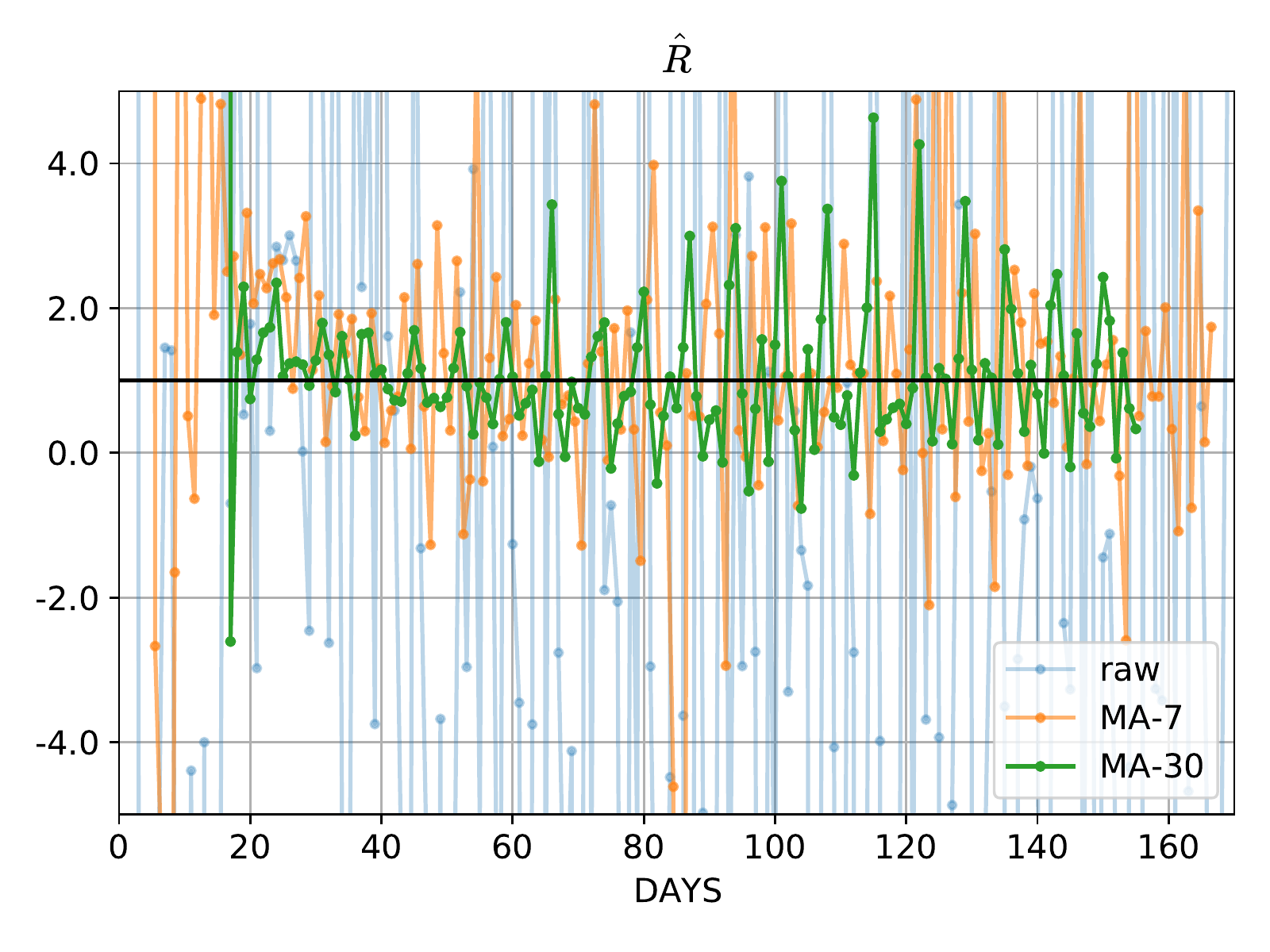}
  \caption{\label{fig:2}Unconstrained estimation of $R$ produced by
    the filter~\eqref{eq:05}--\eqref{eq:008f} corresponding to the
    United States deaths in Fig.~\ref{fig:1} based on raw and centered
    moving average with periods of $7$ and $30$ days. All estimates
    display wide excursion range that take negative and large positive
    values.}
\end{figure}


While smoothing the input does improve the quality of the estimates,
it does so at the expense of unavoidable loss of information as well
as additional lags from filtering. Indeed, the authors
of~\cite{FernandezVillaverde2020} explicitly mention that the above
procedure should be fed a smoothed out version of the signal
$y(k)$. Related approaches, such as~\cite{Buckman2020}, also seem to
rely heavily on smoothing.
These difficulties are the main motivation for the
alternative constrained approach to be introduced in the next section.

This section is closed by presenting a property that is not apparent
in~\cite{FernandezVillaverde2020}, given in the following proposition.
\begin{proposition}
  \label{prop:1}
  For $k \geq 2$ the estimate
  \begin{align}
    \hat{R}(k) &= \frac{\hat{z}_1(k) - \hat{z}_1(k+1)}{\gamma
      (\hat{z}_2(k) - \hat{z}_1(k))} 
  \end{align}
  in which $\hat{z}_1$ and $\hat{z}_2$ are the components of the
  output of the time-invariant linear
  filter~\eqref{eq:008a}--\eqref{eq:008d} driven by the
  input~\eqref{eq:008e} is independent of the value of fatality rate,
  $\delta$.
\end{proposition}
\begin{proof}
Use linearity to write
\begin{align*}
  \hat{z}(k) & = \bar{z}(k) - \delta^{-1} \tilde{z}(k)
\end{align*}
in which $\bar{z}$ and $\tilde{z}$ are such that
\begin{multline*}
  \bar{x}(k+1) = A \, \bar{x}(k) + B, \\
  \bar{z}(k) = C \, \bar{x}(k) + D, \quad
  \bar{x}(0) = x_0,
\end{multline*}
and
\begin{multline*}
  \tilde{x}(k+1) = A \, \tilde{x}(k) + B \, y(k), \\
  \tilde{z}(k) = C \, \tilde{x}(k) + D \, y(k), \quad
  \tilde{x}(0) = 0.
\end{multline*}
Then
\begin{align*}
  \hat{R}(k) &= \frac{\bar{z}_1(k) - \bar{z}_1(k+1) - \delta^{-1}
    (\tilde{z}_1(k) - \tilde{z}_1(k+1))}{\gamma \left (\bar{z}_2(k) -
    \bar{z}_1(k) - \delta^{-1} (\tilde{z}_2(k) - \tilde{z}_1(k)) \right
    )}.
\end{align*}
Because $A^i = 0$, $i \geq 2$ and $B + A B = C e + D = e$, where $e
\in \mathbb{R}^2$ is a vector of ones, for any $k \geq 2$,
\begin{align*}
  \bar{z}(k) &= C \bar{x}(k) + D = \bar{x}(k) = A^k x_0 + \sum_{i = 0}^{k-1} A^i B = e.
\end{align*}
Because $\bar{z}_1(k) = \bar{z}_1(k+1) = \bar{z}_2(k) = 1$ it follows
that
\begin{align*}
  \hat{R}(k) &= \frac{- \delta^{-1} (\tilde{z}_1(k) -
    \tilde{z}_1(k+1))}{- \delta^{-1} \gamma (\tilde{z}_2(k) -
    \tilde{z}_1(k))} = \frac{\tilde{z}_1(k) -
    \tilde{z}_1(k+1)}{\gamma (\tilde{z}_2(k) - \tilde{z}_1(k))}
\end{align*}
is independent of $\delta$ for $k \geq 2$.
\end{proof}

\section{Constrained Estimation}
\label{sec:indirect}

Some basic properties of the variables in the SRIDC model were listed
in~\eqref{eq:10}. In this section an alternative method for estimating
the model variables and the time-varying parameter $R(k)$ will be
introduced that allows one to enforce such and other constraints.

Consider the auxiliary linear time-invariant system
\begin{align}
  \label{eq:06a}
  z_1(k+1) &= z_1(k) - u(k) \\
  \label{eq:06b}
  z_2(k+1) &= z_2(k) - \gamma \, z_2(k) + \gamma z_1(k) \\
  \label{eq:06c}
  z_3(k+1) &= z_3(k) - \theta  \, ( z_3(k) - z_2(k) )
\end{align}
in which $u(k)$ is an input to be
determined. Equations~\eqref{eq:04a}--\eqref{eq:04c} and
\eqref{eq:06a}--\eqref{eq:06c} will have the exact same trajectories if
they have the same initial conditions and
\begin{align}
  \label{eq:07}
  u(k) &= \gamma \, R(k) (z_2(k) - z_1(k)).
\end{align}
This means that instead of estimating $R(k)$ from the nonlinear
model~\eqref{eq:04a}--\eqref{eq:04c} it is possible to
use~\eqref{eq:07} to calculate
\begin{align}
  \label{eq:08}
  \hat{R}(k) &= \frac{\hat{u}(k)}{\gamma \left (\hat{z}_2(k) -
    \hat{z}_1(k) \right )},
\end{align}
in which the input $\hat{u}$ and the variables $\hat{z}_1$ and
$\hat{z}_2$ are estimated from the linear time-invariant
model~\eqref{eq:06a}--\eqref{eq:06c}.

As it will be seen shortly, this alternative approach has several
advantages. First, the resulting estimation problem is a convex
problem that can be solved efficiently even with large number of data
points. Second, the basic constraints~\eqref{eq:10} are all linear
constraints that can be easily incorporated to the problem without
compromising convexity. Third, it is possible to add constraints that
will control the range of $R$ and its derivative, $\dot{R}$, as well
as explicitly promote smoothness of the estimates.

In order to arrive at the desired problem formulation first introduce
the estimator
\begin{align}
  \label{eq:08a}
  \hat{x}(k+1) &= A \, \hat{x}(k) + B \, \hat{u}(k)
\end{align}
in which
\begin{multline}
  \label{eq:08b}
  A =
  \begin{bmatrix}
    1 & 0 & 0 \\
    \gamma & (1 - \gamma) & 0 \\
    0 & \theta & (1 - \theta) 
  \end{bmatrix}, \\
  B =
  \begin{bmatrix}
    -1 \\ 0 \\ 0
  \end{bmatrix}, \qquad
  \hat{x}(k) =
  \begin{pmatrix}
    \hat{z}_1(k) \\ 
    \hat{z}_2(k) \\ 
    \hat{z}_3(k) \\ 
  \end{pmatrix}.
\end{multline}
Equations~\eqref{eq:08a}--\eqref{eq:08b} correspond to a state-space
representation of~\eqref{eq:06a}--\eqref{eq:06c} if $\hat{u}(k) =
u(k)$, $\hat{x}(k) = x(k)$, and $\hat{x}(0) = x(0)$. Consider also the
measurement~\eqref{eq:04d} and
\begin{align}
  \label{eq:08c}
  \hat{y}(k) &= 1 + C \, \hat{x}(k), &
  C &=
  \begin{bmatrix}
    0 & 0 & -1
  \end{bmatrix}.
\end{align}
Note that if $\hat{u}(k) = u(k)$ and $\hat{x}(0) = x(0)$ then $\delta
\, \hat{y}(k) = y(k)$. This motivates the introduction of the cost
function
\begin{align}
  \label{eq:11a}
  \phi_y(w) &=
  \frac{1}{N} \sum_{k = 0}^{N-1} q(k) \left ( y(k) - w(k)
  \right )^2
\end{align} 
and the associated optimal estimation problem
\begin{equation}
  \label{eq:11}
  \begin{aligned}
    \min_{\hat{y}, \hat{x}, \hat{u}} \quad & \phi_y(\delta \, \hat{y}) \\
    \text{s.t.} \quad
    & \hat{x}(k+1) = A \hat{x}(k) + B \hat{u}(k) \\
    & \hat{y}(k) = 1 + C \hat{x}(k) \\
    & \hat{u}(k) \geq 0, \quad k = 0, \cdots, N-1 \\
    & (\hat{x}, \hat{u})  \in \Omega    
  \end{aligned}
\end{equation}
in which $\Omega$ is a constraint set that will be detailed below. In
this paper, $\Omega$ will be comprised of linear constraints, hence
Problem~\eqref{eq:11} will be a convex optimization
problem~\cite{boyd:COP:2004}.

The objective of Problem~\eqref{eq:11} is to produce a non-negative
input $\hat{u}(k)$ that minimizes the weighted sum of squares of the
error between the available measurement $y(k)$, in this case the
deaths, and $\delta \, \hat{y}(k)$ produced by the dynamic
model~\eqref{eq:08a}--\eqref{eq:08b}, in the presence of the
additional convex constraints expressed in the set $\Omega$.

The weighting function $q(k)$ can be used to reflect the uncertainty
level of each measurement. For example, under the assumption that the
measurement noise can be modeled as a zero-mean and Gaussian white
noise process a natural choice would be $q(k) = v(k)^{-1}$, where
$v(k)$ is the variance of the measurement error at time
$k$~\cite{anderson:OFI:1979}. The above problem is a variation on a
standard finite-horizon linear quadratic optimal control
problem~\cite{Kwakernaak1972}.

The non-negativity constraint on $\hat{u}(k)$ follows from one of the
basic constraints in~\eqref{eq:10}. Indeed $\hat{u}(k) = \hat{z}_1(k)
- \hat{z}_1(k+1) \geq 0$. The remaining constraints in~\eqref{eq:10}
can be expressed in the form
\begin{multline*}
  \Omega_b = \big \{ (\hat{x}, \hat{u}) : \quad 0 \leq \hat{x}(k)
  \leq 1, \\ F \hat{x}(k) \leq 0, \quad k = 0, \cdots, N \big \},
\end{multline*}
in which the matrix
\begin{align*}
  F &=
  \begin{bmatrix}
    1 & -1 & 0 \\
    0 & 1 & -1
  \end{bmatrix}.
\end{align*}

\subsection{Constraints on $R$ and $\dot{R}$}
\label{sec:rconstr}

Constraints on the estimates of $R$ and $\dot{R}$ can be translated as
constraints on $\hat{x}$ and $\hat{u}$. The constraints discussed in
this section implicitly assume that $z_2 > z_1$, which will be the
case whenever $z_2 - z_1 = x_2 > 0$, that is whenever the number of
infected is still positive.

If lower- and upper-bounds on the value of $R \in [\underline{R},
  \overline{R}]$ are available then the same constraint applied on the
estimate~\eqref{eq:08} can be translated as
\begin{align*}
  \gamma \underline{R} (\hat{z}_2(k) - \hat{z}_1(k)) \leq \hat{u}(k) \leq \gamma \overline{R} \,
  (\hat{z}_2(k) - \hat{z}_1(k))
\end{align*}
which can be represented by the set of linear constraints
\begin{multline*}
  \Omega_R = \big \{ (\hat{x}, \hat{u}) : \quad \gamma \underline{R}
  \, G \hat{x}(k) \leq \hat{u}(k) \leq \gamma \overline{R} \, G
  \hat{x}(k), \\ k = 0, \cdots, N-1 \big \}
\end{multline*}
in which the matrix
\begin{align*}
  G &=
  \begin{bmatrix}
    -1 & 1 & 0 \\
  \end{bmatrix}.
\end{align*}
Note how important is to formulate the estimation problem in terms of
$R$ rather than $\beta$: the equivalent constraints in $\beta$ would be
nonlinear while the ones in $R$ are linear.

It is also useful to constrain the derivative of $R$, that is
$\dot{R}$, which is easier to manipulate in the continuous-time
version of model~\eqref{eq:06a}--\eqref{eq:06c}, namely
\begin{align*}
  \dot{z}_1 &= - u, \\
  \dot{z}_2 &= - \gamma \, ( z_2 - z_1 ), \\
  \dot{z}_3 &= - \theta  \, ( z_3 - z_2 ),
\end{align*}
from which
\begin{align*}
  R &= \frac{u}{\gamma (z_2 - z_1)}, &
  & \text{and} &
  \dot{R} &= \frac{\dot{u} (z_2 - z_1) - (\dot{z}_2 - \dot{z}_1) u}{\gamma (z_2 - z_1)^2}.
\end{align*}
If $R \in [\underline{R}, \overline{R}]$ then
\begin{align*}
  \frac{\dot{u} - \gamma (\dot{z}_2 - \dot{z}_1) \overline{R}}{\gamma
    (z_2 - z_1)}
  \leq
  \dot{R} 
  \leq \frac{\dot{u} - \gamma (\dot{z}_2 - \dot{z}_1) \underline{R}}{\gamma
    (z_2 - z_1)}
\end{align*}
Therefore, since $z_2 > z_1$, if
\begin{multline*}
  (\dot{z}_2 - \dot{z}_1) \overline{R} + \underline{\dot{R}} (z_2 - z_1)
  \leq
  \gamma^{-1} \dot{u} 
  \leq \\
  (\dot{z}_2 - \dot{z}_1) \underline{R} + \overline{\dot{R}} (z_2 - z_1),
\end{multline*}
then $\dot{R} \in [\underline{\dot{R}}, \overline{\dot{R}}]$. These
inequalities can be expressed approximately in terms of the
discrete-time variables in model~\eqref{eq:06a}--\eqref{eq:06c} upon
substituting
\begin{align*}
  \dot{u} &\approx u(k+1) - u(k), \\
  \dot{z}_2 - \dot{z}_1 &\approx z_2(k+1) - z_1(k+1) - z_2(k) + z_1(k)
  \\
  &\qquad \qquad= u(k) + \gamma (z_1(k) - z_2(k)),
\end{align*}
leading to the constraints on the estimates
\begin{multline*}
  (\gamma \overline{R} - \underline{\dot{R}}) \hat{z}_1(k)
  - (\gamma \overline{R} - \underline{\dot{R}}) \hat{z}_2(k) + \overline{R}
  \hat{u}(k) 
  \leq \\
  \gamma^{-1} (\hat{u}(k+1) - \hat{u}(k)) 
  \leq \\
  (\gamma \underline{R} - \overline{\dot{R}}) \hat{z}_1(k)
  - (\gamma \underline{R} - \overline{\dot{R}}) \hat{z}_2(k) + \underline{R}
  \hat{u}(k).
\end{multline*}
An extension is to have bounds on $\dot{R}(k)$ that vary depending on
$k$. This is especially useful to capture the higher uncertainties
associated with the beginning of the pandemic, a period when noisy
date might suggest a wider variation on $R$ and hence its
derivative. Such time dependent constraint can be represented by the
set
\begin{multline*}
  \Omega_{\dot{R}} = \big \{ (\hat{x}, \hat{u}) : \\ \gamma
  \underline{H}(k) \hat{x}(k) + \gamma \overline{R} u(k) \leq
  \hat{u}(k+1) - \hat{u}(k), \\
  \hat{u}(k+1) - \hat{u}(k) \leq \gamma \overline{H}(k) \hat{x}(k) +
  \gamma \underline{R} u(k), \\ k = 0, \cdots, N-1 \big \} 
\end{multline*}
in which 
\begin{align*}
  \underline{H}(k) &=
  \begin{bmatrix}
    (\gamma \overline{R} - \underline{\dot{R}}(k)) & - (\gamma
    \overline{R} - \underline{\dot{R}}(k)) & 0
  \end{bmatrix}, \\
  \overline{H}(k) &=
  \begin{bmatrix}
    (\gamma \underline{R} - \overline{\dot{R}}(k)) & - (\gamma
    \underline{R} - \overline{\dot{R}}(k)) & 0
  \end{bmatrix},
\end{align*}
and $\dot{R}(k) \in [\underline{\dot{R}}(k),
  \overline{\dot{R}}(k)]$. Even though the above constraints, having
been ported from the continuous-time model to the discrete-time model,
are approximations, they are very effective, as it will be illustrated
by examples later.

\subsection{Initial Condition}
\label{sec:initial-condition}

It is not necessary to have an estimate of the initial condition,
$\hat{x}(0)$, to solve Problem~\eqref{eq:11}. The optimal solution
will provide a suitable estimate of the initial condition.  However,
it is interesting to note that it is always possible to chose
$\hat{x}(0)$ so that $\hat{y}(k) = \delta^{-1} y(k)$ for $k = \{ 0, 1, 2
\}$ without further constraining $\hat{u}(k)$. Indeed, verify that
\begin{align*}
  C B = C A B = 0
\end{align*}
and that
\begin{align*}
  \begin{bmatrix}
    C  \\ C A \\ C A^2 
  \end{bmatrix}
  x(0)
  &= 
  \delta^{-1}
  \begin{pmatrix}
    y(0) \\
    y(1) \\
    y(2)
  \end{pmatrix}
  - e,
\end{align*}
in which $e \in \mathbb{R}^3$ is a vector of ones.  Because the above
coefficient matrix is the Observability Matrix~\cite{Kwakernaak1972}
associated with the pair $(A, C)$ which, for $A$ and $C$
from~\eqref{eq:08b} and~\eqref{eq:08c}, is square and non-singular, one
can calculate
\begin{align}
  \label{eq:mu0}
  \mu_0
  &= 
  \begin{bmatrix}
    C  \\ C A \\ C A^2 
  \end{bmatrix}^{-1}
  \begin{pmatrix}
    \delta^{-1} y(0) - 1 \\
    \delta^{-1} y(1) - 1 \\
    \delta^{-1} y(2) - 1
  \end{pmatrix}.
\end{align}
As discussed above, fixing $\hat{x}(0) = \mu_0$ may not lead to the
best possible overall estimate but one could incorporate such
knowledge, if desired, by adding the function
\begin{align}
  \label{eq:phiinitial}
  \phi_0(\hat{x}(0)) = \| \Pi (\hat{x}(0) - \mu_0) \|_2^2
\end{align}
to the cost of Problem~\eqref{eq:11}. Matrix~$\Pi$ can be used to
weigh the user confidence on the estimate $\mu_0$. The examples shown
in the present paper do not make use
of~\eqref{eq:phiinitial}. However, weighing prior knowledge on the
initial condition can be useful in the presence of additional
measurements, to be discussed in Section~\ref{sec:extensions}.

\subsection{Smoothness Cost}
\label{sec:smoothness}

The estimation Problem~\eqref{eq:11} takes the form of a finite
horizon optimal control problem~\cite{Kwakernaak1972}. However, a
typical finite horizon optimal control problems is often formulated
with two more types of costs: a penalty on the terminal state and a
direct penalty on the control cost, typically a measure of the energy
of the signal $\hat{u}(k)$. There are lots of good reasons for such
penalties to be part of the cost~\cite{Kwakernaak1972}. Here a penalty
on the signal $\hat{u}$ will be used as a way to promote smoothness of
the estimates.

Consider first a penalty on the terminal state. As it is typical of
discrete-time dynamic systems, the effects of the input signal
$\hat{u}$ may not appear in the output signal $\hat{y}$ until a number
of iterations has taken place. In the case of the
model~\eqref{eq:08a}--\eqref{eq:08c}, since $B C = B A C = 0$, it
takes at least two iterations for the input to show up at the
output. That is, the value of the input $\hat{u}(k)$ will only appear
in the output $\hat{y}(k+2)$. See also the discussion in
Section~\ref{sec:initial-condition}. This means that the final two
values of $\hat{u}(k)$ have no effect on the cost function of
Problem~\eqref{eq:11}. However, they will have an effect on the state,
which ultimately affects the estimate $\hat{R}(k)$. This means that
the last two estimates of $\hat{R}(k)$ should probably not be
trusted. This is equivalent to the two-step delay of the unconstrained
estimator discussed earlier in Section~\ref{sec:direct}.

In a typical control problem, a terminal cost ensures that, even in
the absence of measurements that can help determine $u(k)$ on those
final instants, the final state is steered toward a \emph{desired}
state. However, in Problem~\eqref{eq:11} there does not seem to be a
clear choice of a desired state, unless the analysis pertains to past
events in which the disease has already reached equilibrium and
information on the equilibrium state is available. For this reason, in
the context of the COVID-19 epidemic, no terminal
constraint shall be imposed.

As for a running cost on the signal $\hat{u}$, solutions to
Problem~\eqref{eq:11} are likely to still produce a signal
$\hat{u}(k)$ that can have significant variations, even after imposing
the constraints discussed so far. Indeed, the very nature of the
Problem~\eqref{eq:11} is to produce an optimal $\hat{u}(k)$ that will
do its best to capture the variations implied by a
changing~$y(k)$. However, as far as estimating $R(k)$, it may not be
important to capture every single variation, but rather to smooth out
the trends. This goal is achieved by adding the smoothness cost
\begin{multline}
  \label{eq:smooth}
  \phi_s(\hat{u}) = \\
  \frac{1}{N \!-\! 2} \left ( r(0) | \hat{u}(0) |^2 + \!
  \sum_{k = 1}^{N-3} r(k) \left | \hat{u}(k) -
  \hat{u}(k-1) \right |^2 \right )
\end{multline}
which penalizes the total variations of $\hat{u}$ measured at
consecutive samples. This cost function promotes smoothness of
$\hat{u}(k)$, which in turns promotes smoothness of $\hat{R}(k)$, as
it will be seen in the examples. Note also that~\eqref{eq:smooth} does
not penalize the last two values of $\hat{u}(k)$ since, as discussed
above, they do not affect the measurements at $k = N-2$ and $N-1$.

\subsection{Trading-off Accuracy Versus Smoothness}
\label{sec:tradeoff}

\begin{figure}
  \includegraphics[width=\columnwidth]{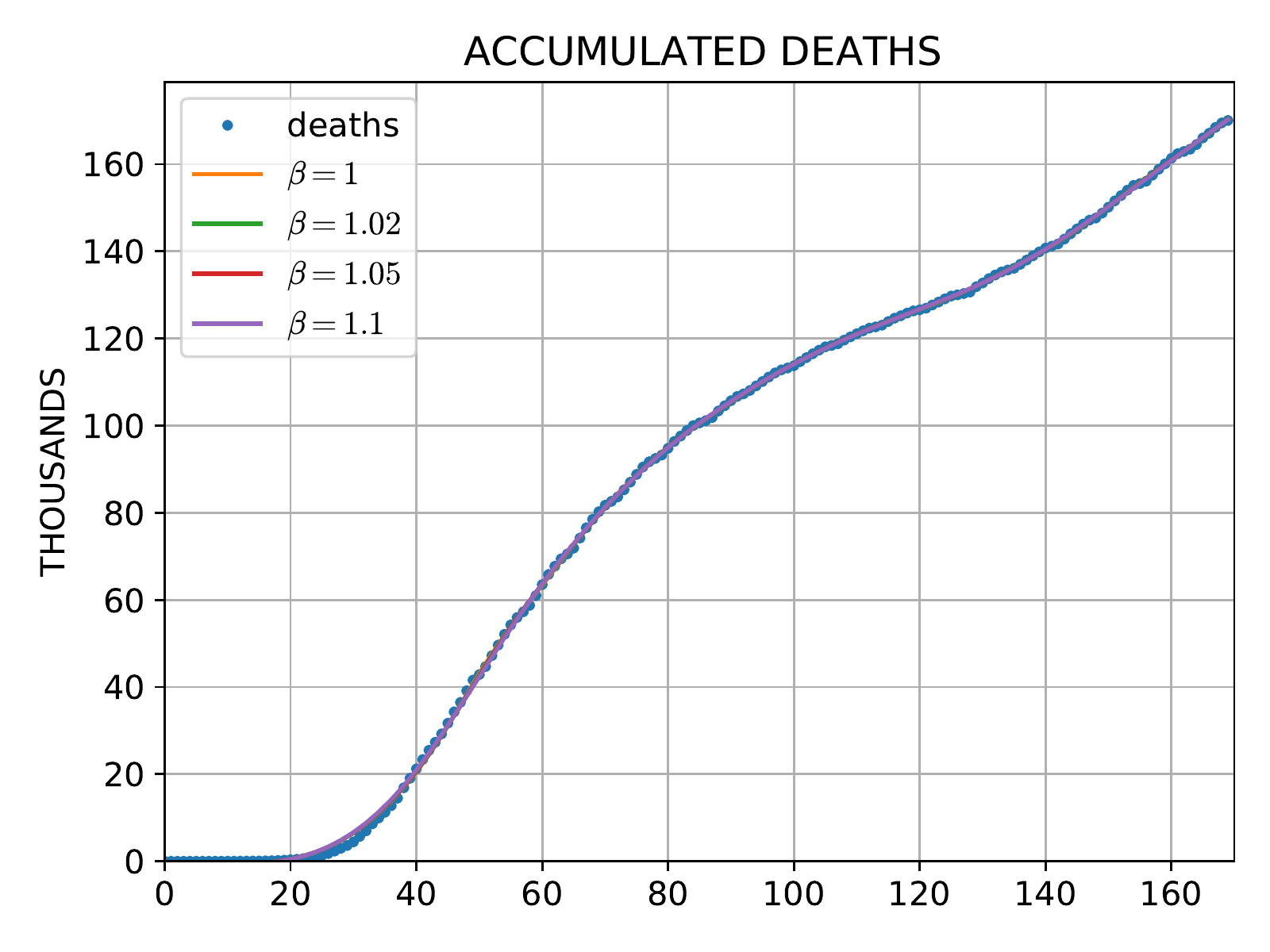}

  \hfill
  \includegraphics[width=\columnwidth]{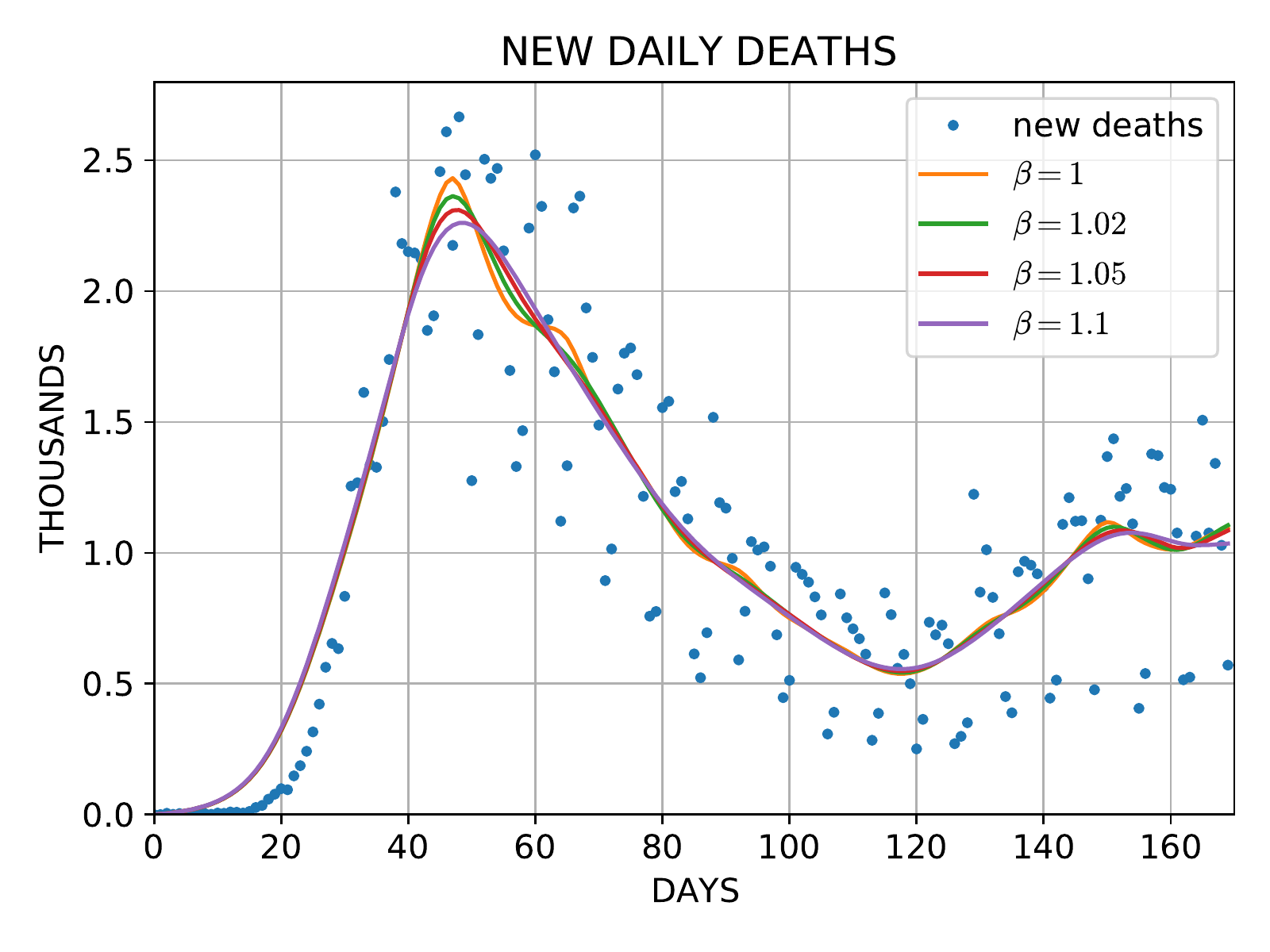}
  
  \caption{\label{fig:5}Accumulated and new daily deaths attributed to
    COVID--19 in the United States from 01/22/2020 through
    08/16/2020~\cite{Dong2020}. Day $0$ corresponds to the first day
    in this range in which a non-zero number of deaths was
    reported. Also shown is the corresponding constrained estimates
    produced by solving Problem~\eqref{eq:15} for various values of
    $\beta$. The solution to Problem~\eqref{eq:11} corresponds to
    $\beta = 1$.}
\end{figure}

\begin{figure}
  \includegraphics[width=\columnwidth]{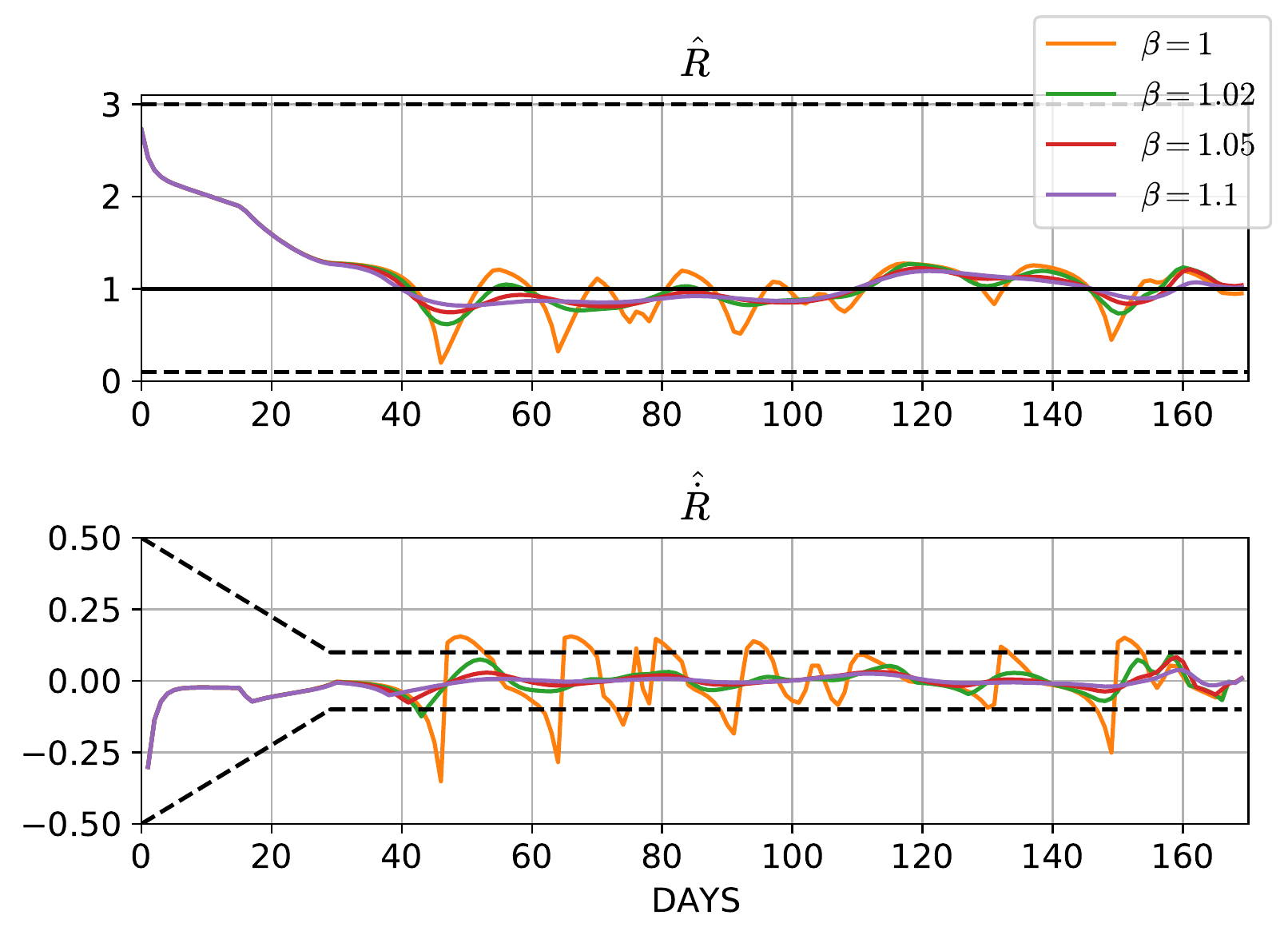}
 
  \caption{\label{fig:6}Constrained estimates produced by solving
    Problem~\eqref{eq:15} for various values of $\beta$ for the data
    in Fig.~\ref{fig:5}. The constraints imposed on the estimated $R$
    and $\dot{R}$ are shown by the dashed lines.}
\end{figure}

\begin{figure}
  \centering
  \includegraphics[width=\columnwidth]{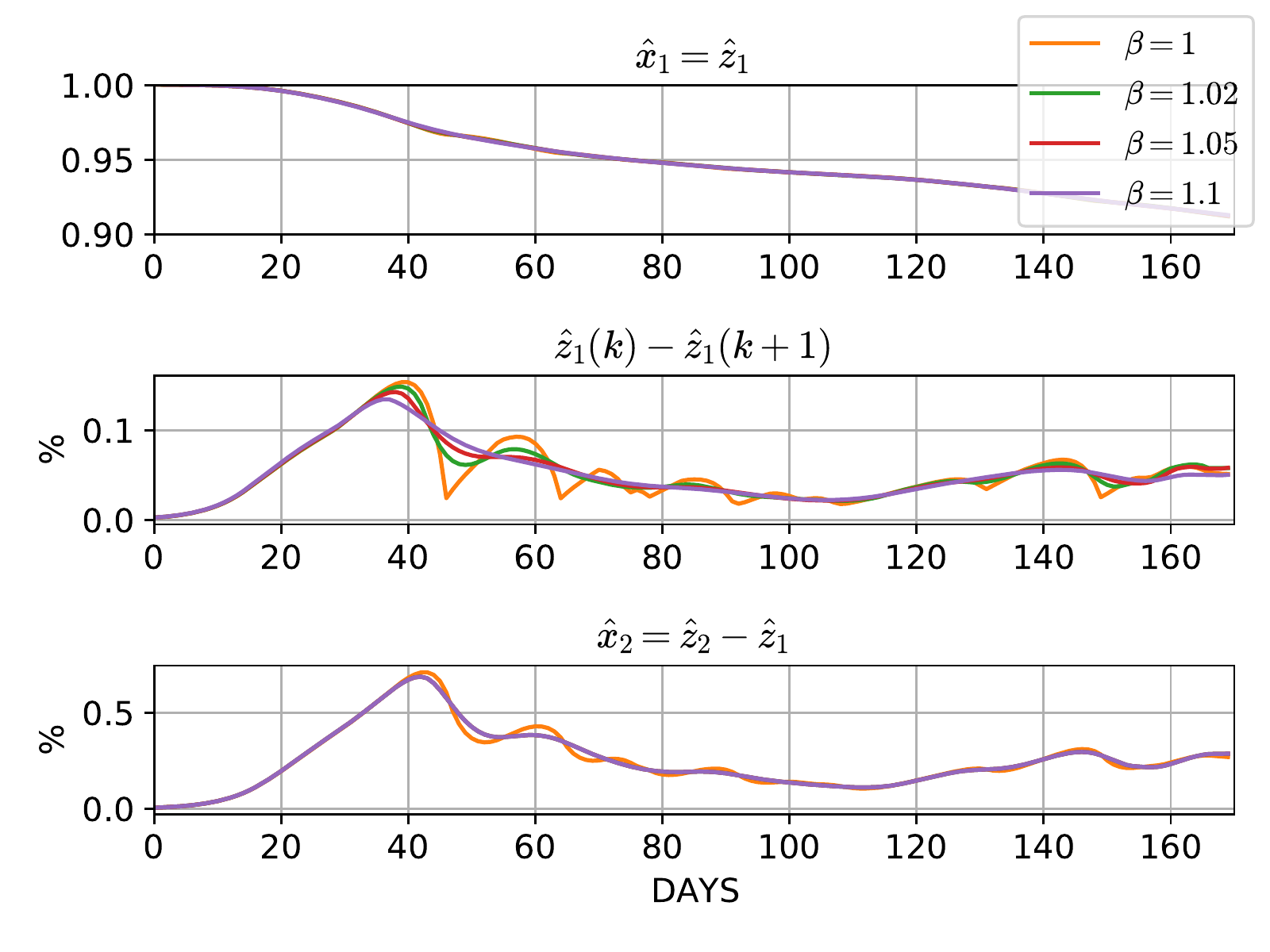}
  \caption{\label{fig:8}Constrained estimates $\hat{x}_1 = \hat{z}_1$
    and $\hat{x}_2 = \hat{z}_2 - \hat{z}_1$ produced by solving
    Problem~\eqref{eq:15} for various values of $\beta$ for the data
    in Fig.~\ref{fig:5}. Note how $\hat{x}_1 =
    \hat{z}_1$ is monotonically decreasing and remains below one, and
    $\hat{x}_2 = \hat{z}_2 - \hat{z}_1$ and $\hat{z}_1(k) -
    \hat{z}_1(k+1)$ remain positive.}
\end{figure}

This section will illustrate how the optimization problem and
the constraints discussed so far can be used to produce smooth
estimates of the SIRDC model variables and the parameter $R$. Let us
start by solving Problem~\eqref{eq:11} for the United States data
shown before in Fig.~\ref{fig:1}. The following constant parameters
were used:
\begin{align*}
  \gamma &= 0.2, &
  \theta &= 0.1, &
  \delta &= 0.065, &
  \overline{R} &= 3, &
  \underline{R} &= 0.1,
\end{align*}
with the constraint set
\begin{align*}
  \Omega = \Omega_b \cap \Omega_{R} \cap \Omega_{\dot{R}}.
\end{align*}
The derivative of $R$, $\dot{R}$, was constrained by the
time-dependent bounds
\begin{align*}
  \underline{\dot{R}}(k) &= - \overline{\dot{R}}(k), \\
  \overline{\dot{R}}(k) &= \begin{cases}
    0.5 (k/30) + 0.1 (1 - k/30), & k \leq 30, \\
    0.1, & k > 30,
  \end{cases}
\end{align*}
which allows larger variations at the beginning of the epidemic, and a
constant weight
\begin{align*}
  q(k) &= 1, & k &= 0, \cdots, N-1,
\end{align*}
was used in the cost function. No penalty on the initial condition was
imposed.

The estimated deaths and new daily deaths and the corresponding
estimates for $R$ and $\dot{R}$ obtained by Problem~\eqref{eq:11} are
shown in Fig.~\ref{fig:5} and~\ref{fig:8}, with the label
$\beta = 1$. Note how the constraints on $R$ are enforced at all times
while the constraints on $\dot{R}$ are approximately enforced, as
discussed in Section~\ref{sec:rconstr}. All numerical examples in this
paper were formulated using CVXPY~\cite{Diamond2016} and solved using
MOSEK's conic solver~\cite{Andersen2003}.

Enforcing the constraints on the model variables and parameters during
the estimation process ensures that the estimates produced are much
better behaved and smoother when compared with the estimates obtained
by the unconstrained estimation methods of
Section~\ref{sec:direct}. The smoothness of the estimate can be
further enhanced by incorporating a smoothness cost as discussed in
Section~\ref{sec:smoothness}. As it is customary, one could modify
Problem~\eqref{eq:11} by replacing its cost function by
\begin{align*}
  \phi_y(\delta \, \hat{y}) + \eta \, \phi_s(\hat{u})
\end{align*}
where $\eta > 0$ is a penalty parameter. The correct tuning of the
parameter $\eta$ can however be tricky. Instead, perform the
following two step procedure:
\begin{enumerate}
\item Solve the convex optimization Problem~\eqref{eq:11} and
  determine its global optimal solution and cost. Let $\rho^*$ be the
  minimal cost.
\item Select $\beta \geq 1$ and solve the convex quadratic
  optimization problem with linear and convex quadratic constraints
  \begin{equation}
    \label{eq:15}
    \begin{aligned}
      \min_{\hat{y}, \hat{x}, \hat{u}} \quad & \delta^2 \, \phi_s(\hat{u}) \\
      \text{s.t.} \quad
      & \hat{x}(k+1) = A \hat{x}(k) + B \hat{u}(k) \\
      & \hat{y}(k) = C \hat{x}(k) + E r(k) \\
      & \hat{u}(k) \geq 0, \quad k = 0, \cdots, N-1 \\
      & (\hat{x}, \hat{u})  \in \Omega \\
      & \phi_y(\delta \, \hat{y}) \leq \beta \, \rho^*
    \end{aligned}
  \end{equation}
\end{enumerate}
The reason for scaling the cost function by the square of the fatality
rate $\delta$ will be made clear in Section~\ref{sec:fatality}.

The parameter $\beta \geq 1$ can be interpreted as how much accuracy
one is willing to trade for smoothness. Indeed when $\beta = 1$,
Problems~\eqref{eq:11} and~\eqref{eq:15} admit the exact same optimal
solution. However, as $\beta$ increases, smoother solutions are
possible at the expense of a higher estimation error. In the case of
the United States COVID-19 data, the estimates produced with $\beta =
\{ 1.02, 1.05, 1.1 \}$ are also shown in
Figs.~\ref{fig:5}--\ref{fig:8}.

The impact of the value of $\beta$ is still data dependent. Indeed,
the more noise is present in the data the less smooth one would expect
the solution to Problem~\eqref{eq:11} to be, and the higher one might
need to set $\beta$ for the desired level of smoothness.  Note also
that as the penalty $\beta$ increases the constraints on $R$ and its
derivative becomes less and less active. However, solving
Problem~\ref{eq:11} without these constraints would make the choice of
$\beta$ much more difficult, as the value of the cost function of
Problem~\ref{eq:11} is allowed to be reduced further by increasingly
less smooth solutions. By enforcing these constraints earlier in
Problem~\eqref{eq:11} it is found that a choice of $\beta \in [1, 2]$
is enough to produce suitable solutions for data from diverse
countries, to be presented in Section~\ref{sec:discussion}.

\subsection{Independence of the Fatality Rate}
\label{sec:fatality}

The constrained estimator obtained as the solution to the optimization
problems~\eqref{eq:11} and~\eqref{eq:15} enjoy the same independence
of the fatality rate, $\delta$, as the unconstrained estimator from
Section~\ref{sec:direct}. The following result is analogous to
Proposition~\ref{prop:1}. Its proof reveals the need for the scaling of
the objective function in Problem~\eqref{eq:15}.
\begin{proposition}
  \label{prop:2}
  The estimate
  \begin{align}
    \hat{R}(k) &= \frac{\hat{u}(k)}{\gamma
      (\hat{z}_2(k) - \hat{z}_1(k))} 
  \end{align}
  in which $\hat{z}_1$ and $\hat{z}_2$ are obtained as solutions to
  Problem~\eqref{eq:11} or Problem~\eqref{eq:15} is independent of the
  value of the fatality rate, $\delta$.
\end{proposition}
\begin{proof}
  Let $\hat{y}_\delta$, $\hat{x}_\delta$ and $\hat{u}_\delta$ be the 
  optimal solution to Problem~\eqref{eq:11} for some $\delta > 0$,
  that is
  \begin{align*}
    (\hat{y}_\delta, \hat{x}_\delta, \hat{u}_\delta) &= \arg
    \min_{(\hat{y}, \hat{x}, \hat{u}) \in \hat{\Omega}}
    \phi_y(\delta \, \hat{y})
  \end{align*}
  in which
  \begin{multline*}
    \hat{\Omega} = \big \{ (\hat{y}, \hat{x}, \hat{u}) : \quad
    \hat{x}(k+1) = A \hat{x}(k) + B \hat{u}(k), \\
    \hat{y}(k) = 1 + C \hat{x}(k), \quad   
    \hat{u}(k) \geq 0, \quad k = 1, \cdots, N-1, \\
    (\hat{x}, \hat{u}) \in \Omega = \Omega_b \cap
    \Omega_R \cap \Omega_{\dot{R}} \big \},
  \end{multline*}
  for $A$, $B$, and $C$ from~\eqref{eq:08b}--\eqref{eq:08c}. Since
  $\delta$ only affects the value of the cost function, there will
  exist an optimal solution for any $\delta > 0$ as long as
  $\hat{\Omega} \neq \emptyset$.

  Now let $0 < \lambda \leq 1$, and calculate
  \begin{multline*}
    \hat{x}_\lambda(k+1) = A \, \hat{x}_\lambda(k) + B \, \hat{u}_\lambda(k), \\
    \bar{x}_\lambda(0) = \lambda \, \hat{x}_\delta(0) + (1 -
    \lambda) \, e, \quad
    \hat{u}_{\lambda}(k) = \lambda \, \hat{u}_{\delta}(k),
  \end{multline*}
  in which $e \in \mathbb{R}^3$ is a vector of ones, and use the fact
  that $A^k e = e$, $k = 0, 1, \cdots$, to show that
  \begin{align*}
    \hat{x}_{\lambda}(k)
    &= \lambda \,
    \hat{x}_\delta(k) + (1 - \lambda) \, e.
  \end{align*}
  Since $0 \leq \hat{x}_\delta(k) \leq e$ and $0 < \lambda \leq 1$
  \begin{align*}
    0 \leq \hat{x}_{\lambda}(k) \leq e.
  \end{align*}
  Furthermore $F e = 0$ so that
  \begin{align*}
    F \hat{x}_{\lambda}(k) = \lambda \, F \hat{x}_\delta(k) \leq 0
  \end{align*}
  and $\hat{x}_\lambda(k) \in \Omega_b$. Likewise, because $G \, e = 0$,
  $\underline{H}(k) \, e = \overline{H}(k) \, e = 0$, $\hat{x}_\lambda(k)
  \in \Omega_R \cap \Omega_{\dot{R}}$. Finally, using the fact that $C
  \, e = -1$,
  \begin{multline*}
    \hat{y}_{\lambda}(k)
    = 1 + C \hat{x}_{\lambda}(k) =
    \\ 1 + (1 - \lambda) \, C e + \lambda \, C \hat{x}_\delta(k)
    = \lambda \, \hat{y}_\delta(k),
  \end{multline*}
  from which one concludes that $(\hat{y}_\lambda, \hat{x}_\lambda,
  \hat{u}_\lambda) \in \hat{\Omega}$.

  Note that for any $y$, $z$ and $0 < \lambda \leq 1$ it is true that
  \begin{align*}
    N \phi_{y}(\lambda^{-1} z) &= 
    \| y - \lambda^{-1} z \|_Q^2
    \\
    &\geq \lambda^{-1} \left ( \| z \|_Q^2 - 2 \, z^T Q y \right ) + \| y \|_Q^2 
    \\
    &\geq \| y - z \|_Q^2
    = N \phi_y(z),
  \end{align*}
  in which $Q = \operatorname{diag}(q(0),
  \cdots, q(N-1))$ and $\| x \|_Q^2 = x^T Q \, x$, so that
  \begin{align*}
    \min_{(\hat{y}, \hat{x}, \hat{u}) \in \hat{\Omega}}
    \phi_{y}(\lambda^{-1} \delta \, \hat{y}) 
    &\geq
    \min_{(\hat{y}, \hat{x}, \hat{u}) \in \hat{\Omega}}
    \phi_y(\delta \, \hat{y})
    =
    \phi_y(\delta \, \hat{y}_\delta).
  \end{align*}
  On the other hand, since $(\hat{y}_\lambda, \hat{x}_\lambda,
  \hat{u}_\lambda) \in \hat{\Omega}$,
  \begin{align*}
    \min_{(\hat{y}, \hat{x}, \hat{u}) \in \hat{\Omega}}
    \phi_{y}(\lambda^{-1} \delta \, \hat{y})
    \leq 
    \phi_{y}(\lambda^{-1} \delta \, \hat{y}_\lambda)
    =
    \phi_y(\delta \, \hat{y}_\delta).
  \end{align*}
  Combining these two inequalities it is possible to conclude that
  \begin{align*}
    \min_{(\hat{y}, \hat{x}, \hat{u}) \in \hat{\Omega}}
    \phi_{y}(\lambda^{-1} \delta \, \hat{y}) 
    =
    \phi_y(\delta \, \hat{y}_\delta), \quad \text{for all } 0 < \lambda
    \leq 1,
  \end{align*}
  which proves the proposition for Problem~\eqref{eq:11} since the
  above discussion holds for any small enough $\delta > 0$.

  As for Problem~\eqref{eq:15}, the above discussion means that the
  constraint
  \begin{align*}
    \phi_{y}(\delta \, \hat{y}) \leq \beta \, \rho^*
  \end{align*}
  will be unaffected by the choice of $\delta$ since $\rho^*$ is
  independent of $\delta$. Therefore, an argument similar to the
  one used for Problem~\eqref{eq:15} in which
  \begin{multline*}
    (\hat{y}_\delta, \hat{x}_\delta, \hat{u}_\delta) = \arg
    \min_{(\hat{y}, \hat{x}, \hat{u}) \in \hat{\Omega}_s} \delta^2 \,
    \phi_s(\hat{u}), \\ \hat{\Omega}_s = \{ (\hat{y}, \hat{x},
    \hat{u}) \in \hat{\Omega} : \quad \phi_{y}(\delta \, \hat{y}) \leq
    \beta \, \rho^* \},
  \end{multline*}
  leads to
  \begin{align*}
    \min_{(\hat{y}, \hat{x}, \hat{u}) \in \hat{\Omega}_s} \! \lambda^{-2}
    \delta^2 \, \phi_{s}(\hat{u}) \geq \! \min_{(\hat{y}, \hat{x}, \hat{u})
      \in \hat{\Omega}_s} \! \delta^2 \, \phi_{s}(\hat{u}) = \delta^2 \, \phi_{s}(\hat{u}_\delta)
  \end{align*}
  for all $0 < \lambda \leq 1$ and, since $(\hat{y}_\lambda,
  \hat{x}_\lambda, \hat{u}_\lambda) \in \hat{\Omega}_s$, 
  \begin{align*}
    \min_{(\hat{y}, \hat{x}, \hat{u}) \in \hat{\Omega}_s} \! \lambda^{-2}
    \delta^2 \, \phi_{s}(\hat{u}) \leq \lambda^{-2} \delta^2 \,
    \phi_{s}(\hat{u}_\lambda) = \delta^2 \, \phi_{s}(\hat{u}_\delta),
  \end{align*}
  from which
  \begin{align*}
    \min_{(\hat{y}, \hat{x}, \hat{u}) \in \hat{\Omega}_s} \! \lambda^{-2}
    \delta^2 \, \phi_{s}(\hat{u}) = \delta^2 \, \phi_{s}(\hat{u}_\delta),
  \end{align*}
  for all $0 < \lambda \leq 1$, as in Problem~\eqref{eq:11}.
\end{proof}

\subsection{Problem Summary}
\label{sec:summary}

The optimization problems~\eqref{eq:11} and~\eqref{eq:15} are convex
quadratic programs with linear and convex quadratic constraints. It is
possible to take advantage of the linear nature of the
equations~\eqref{eq:08a}--\eqref{eq:08b} to propagate the state
evolution as a function of the inputs and the initial condition. That
is the entire state $\hat{x}(k)$, $k = 0, \cdots, N$, can be written
as
\begin{align*}
  \begin{pmatrix}
    \hat{x}(0) \\
    \hat{x}(1) \\
    \vdots \\
    \hat{x}(N)
  \end{pmatrix}
  &=
  \mathsf{T} \, \mathsf{x}, &
  \mathsf{x} &=
  \begin{pmatrix}
    \hat{x}_0 \\
    \hat{u}(0) \\
    \hat{u}(1) \\
    \vdots \\
    \hat{u}(N-1)
  \end{pmatrix}
\end{align*}
in which
\begin{align*}
  \mathsf{T} &= 
  \begin{bmatrix}
    I & 0 & 0 & \cdots & 0 \\
    A & B & 0 & \cdots & 0 \\
    \vdots & \vdots & \vdots & \ddots & \vdots \\
    A^{N} & A^{N-1} B & A^{N-2} B & \cdots & B
  \end{bmatrix}.
\end{align*}
Using the above, one can write
\begin{align*}
  \phi_0(\hat{x}(0)) +
  \phi_y(\delta \, \hat{y}) = \| \mathsf{Q}^{1/2} \left ( \mathsf{A} \,
  \mathsf{x} - \mathsf{y} \right ) \|_2^2
\end{align*}
in which $\mathsf{Q} = \operatorname{diag}(q(0), \cdots, q(N-1))$ and
\begin{align*}
  \mathsf{y} &=
  \begin{pmatrix}
    \Pi \, \mu_0 \\
    y(0) - \delta \\
    y(1) - \delta \\
    \vdots \\
    y(N-1) - \delta
  \end{pmatrix}, \\
  \mathsf{A} &=
  \begin{bmatrix}
    \Pi & 0 & 0 & \cdots & 0 \\
    \delta \, C & 0 & 0 & \cdots & 0 \\
    \delta \, C A & \delta \, C B & 0 & \cdots & 0 \\
    \vdots & \vdots & \vdots & \ddots & \vdots \\
    \delta \, C A^{N-1} & \delta \, C A^{N-2} B & \delta \, C A^{N-3} B & \cdots & 0
  \end{bmatrix}.
\end{align*}
The cost function of Problem~\eqref{eq:11} is a special case in which
$\Pi = 0$.

Similar manipulations can convert the linear constraints $\Omega_b$
and $\Omega_R$ to the form 
\begin{align*}
  0 \leq \mathsf{T} \, \mathsf{x} &\leq 1, &
  \mathsf{F} \, \mathsf{x} &\leq 0, &
  \gamma \, \underline{R} \, \mathsf{G} \, \mathsf{x} &\leq
  \mathsf{I} \, \mathbf{x} \leq
  \gamma \, \overline{R} \, \mathsf{G} \, \mathsf{x},
\end{align*}
in which $\mathsf{I} = \begin{bmatrix} 0 & I \end{bmatrix}$, 
\begin{align*}
  \mathsf{F} &=
  \begin{bmatrix}
    F & 0 & 0 & \cdots & 0 \\
    F A & F B & 0 & \cdots & 0 \\
    \vdots & \vdots & \vdots & \ddots & \vdots \\
    F A^{N} & F A^{N-1} B & F A^{N-2} B & \cdots & F B
  \end{bmatrix}, \\
  \mathsf{G} &=
  \begin{bmatrix}
    G & 0 & 0 & \cdots & 0 \\
    G A & G B & 0 & \cdots & 0 \\
    \vdots & \vdots & \vdots & \ddots & \vdots \\
    G A^{N-1} & G A^{N-2} B & G A^{N-3} B & \cdots & 0
  \end{bmatrix}.
\end{align*}
The constraints in $\Omega_{\dot{R}}$ can be reformulated as
\begin{align*}
  \gamma \, \underline{\mathsf{H}} \, \mathsf{x} &\leq
  \mathsf{D} \, \mathbf{x} \leq
  \gamma \, \overline{\mathsf{H}} \, \mathsf{x}
\end{align*}
in which
\begin{align*}
  \underline{\mathsf{H}} &=
  \begin{bmatrix}
    \underline{H}(0) & \overline{R} & \cdots & 0 \\
    \underline{H}(1) A & \underline{H}(1) B & \cdots & 0 \\
    \vdots & \vdots & \ddots & \vdots \\
    \underline{H}(N-1) A^{N-1} & \underline{H}(N-1) A^{N-2} B & \cdots & \overline{R}
  \end{bmatrix}, \\
  \overline{\mathsf{H}} &=
  \begin{bmatrix}
    \overline{H}(0) & \underline{R} & \cdots & 0 \\
    \overline{H}(1) A & \overline{H}(1) B & \cdots & 0 \\
    \vdots & \vdots & \ddots & \vdots \\
    \overline{H}(N-1) A^{N-1} & \overline{H}(N-1) A^{N-2} B & \cdots & \underline{R}
  \end{bmatrix}, \\
  \mathsf{D} &=
  \begin{bmatrix}
    0 & -1 & 1 & \cdots & 0 \\
    0 & 0 & -1 & \cdots & 0 \\
    \vdots & \vdots & \vdots & \ddots & \vdots \\
    0 & 0 & 0 & \cdots & 1
  \end{bmatrix}.
\end{align*}

The above can be put together to reformulate Problem~\eqref{eq:11} as
the following convex quadratic program with linear constraints
\begin{align*}
  \min_{\mathsf{x}} \quad & \| \mathsf{Q}^{1/2} \left (
  \mathsf{A} \, 
  \mathsf{x} - \mathsf{y} \right ) \|_2^2 \\
  \text{s.t.} \quad & 0 \leq \mathsf{T} \, \mathsf{x} \leq 1, \quad
  \mathsf{F} \, \mathsf{x} \leq 0, \\
  & \gamma \, \underline{R} \, \mathsf{G} \, \mathsf{x} \leq
  \mathsf{I} \, \mathbf{x} \leq \gamma \, \overline{R} \, \mathsf{G}
  \, \mathsf{x}, \\
  & \gamma \, \underline{\mathsf{H}} \, \mathsf{x} \leq
  \mathsf{D} \, \mathbf{x} \leq
  \gamma \, \overline{\mathsf{H}} \, \mathsf{x}
\end{align*}

Likewise, the cost $\phi_s$ can be expressed as
\begin{align*}
  \phi_s(\hat{u}) &= \| \mathsf{R}^{1/2} \, \mathsf{B} \, \mathsf{x} \|_2^2
\end{align*}
in which $\mathsf{R} = \operatorname{diag}(r(0), \cdots, r(N-1))$,
\begin{align*}
  \mathsf{B} &=
  \begin{bmatrix}
    0 & 1 & 0 & \cdots & 0 \\
    0 & -1 & 1 & \cdots & 0 \\
    \vdots & \vdots & \vdots & \ddots & \vdots \\
    0 & 0 & 0 & \cdots & 1
  \end{bmatrix}.
\end{align*}
This means that problem~\eqref{eq:15} can be formulated as the convex
quadratic program with quadratic constraints
\begin{align*}
  \min_{\mathsf{x}} \quad & \delta^2 \, \| \mathsf{R}^{1/2} \, \mathsf{B} \,
  \mathsf{x} \|_2^2 \\
  \text{s.t.} \quad & 0 \leq \mathsf{T} \, \mathsf{x} \leq 1, \quad
  \mathsf{F} \, \mathsf{x} \leq 0, \\
  & \gamma \, \underline{R} \, \mathsf{G} \, \mathsf{x} \leq
  \mathsf{I} \, \mathbf{x} \leq \gamma \, \overline{R} \, \mathsf{G}
  \, \mathsf{x}, \\
  & \gamma \, \underline{\mathsf{H}} \, \mathsf{x} \leq
  \mathsf{D} \, \mathbf{x} \leq
  \gamma \, \overline{\mathsf{H}} \, \mathsf{x}, \\
  & \| \mathsf{Q}^{1/2} \left ( \mathsf{A} \, \mathsf{x} - \mathsf{y}
  \right ) \|_2^2 \leq \beta \, \rho^*
\end{align*}
The above problems can be formulated and solved efficiently using
modern convex optimization algorithms in stock desktop computers for
problems with tens of thousands of variables and constraints.

\subsection{Handling Additional Measurements}
\label{sec:extensions}

The proposed constrained optimization approach can be extended to
handle additional measurements. For example, one could leverage
testing data to estimate the current number of infected and resolving
cases, that is to provide additional measurements of the model
variables $x_2$ and $x_3$. By grouping these measurements into a
vector $y(k) \in \mathbb{R}^3$ in which the first entry is the
fraction of deaths, the second entry is the fraction of infected
individuals, and the third entry is the fraction of resolving
individuals, one can calculate the best constrained estimate by using
the exact same estimator dynamic model
as~\eqref{eq:08a}--\eqref{eq:08b} with the extended measurement model
\begin{align*}
  \hat{y}(k) &= g + H \hat{x}(k), &
  \!\!H &=
  \begin{bmatrix}
    0 & 0 & -\delta \\
    -1 & 1 & 0 \\
    0 & -1 & 1
  \end{bmatrix}\!, &
  \!\!g &=
  \begin{pmatrix}
    \delta \\ 0 \\ 0
  \end{pmatrix}\!,
\end{align*}
resulting into the optimization problem
\begin{equation}
  \label{eq:20}
  \begin{aligned}
    \min_{\hat{y}, \hat{x}, \hat{u}} \quad & \phi_y(\hat{y}) \\
    \text{s.t.} \quad
    & \hat{x}(k+1) = A \hat{x}(k) + B \hat{u}(k) \\
    & \hat{y}(k) = g + H \hat{x}(k) \\
    & \hat{u}(k) \geq 0, \quad k = 0, \cdots, N-1 \\
    & (\hat{x}, \hat{u})  \in \Omega    
  \end{aligned}
\end{equation}
Note how the fatality rate has been incorporated into the matrix $H$
and vector $g$. It should not be expected that the independence
property of Proposition~\ref{prop:2} holds in the presence of
the additional measurements since scaling the infected and resolving
population to match a given fatality rate, as done in the proof of
Proposition~\ref{prop:2}, will no longer preserve optimality. In fact,
one might use the additional data to jointly estimate the parameter
$\delta$.

Another possible extension that might be especially useful in the
presence of additional measurement is the relaxation of the dynamic
equality constraints in Problem~\eqref{eq:20} as a penalty function,
as usually done in smoothing
problems~\cite{anderson:OFI:1979}. Additional nonlinear model features
could also be added at the expense of loosing convexity of the overall
optimization problem.

\section{Conclusions and Discussion}
\label{sec:discussion}

\begin{figure}
  \includegraphics[width=\columnwidth]{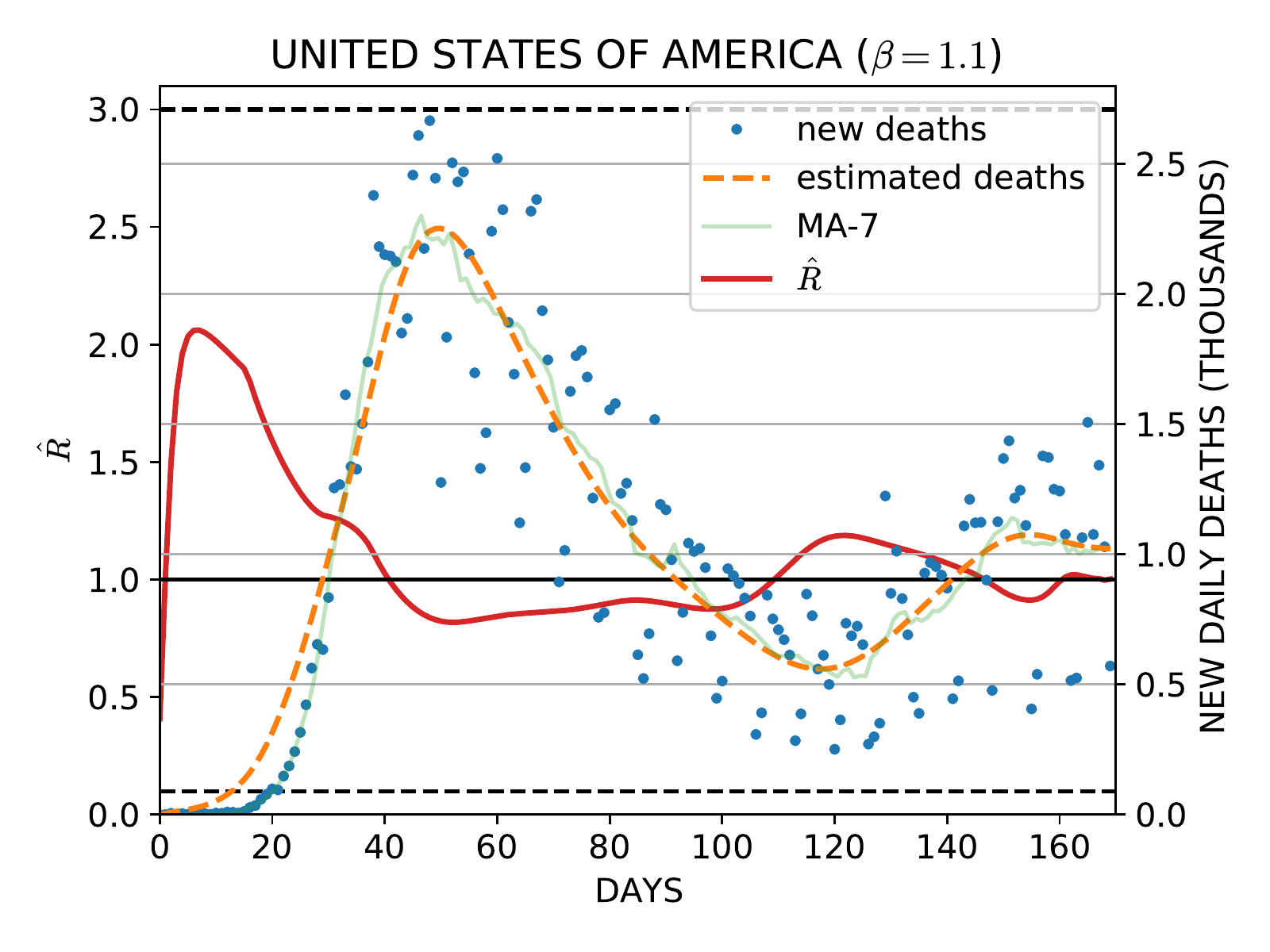}
  \includegraphics[width=\columnwidth]{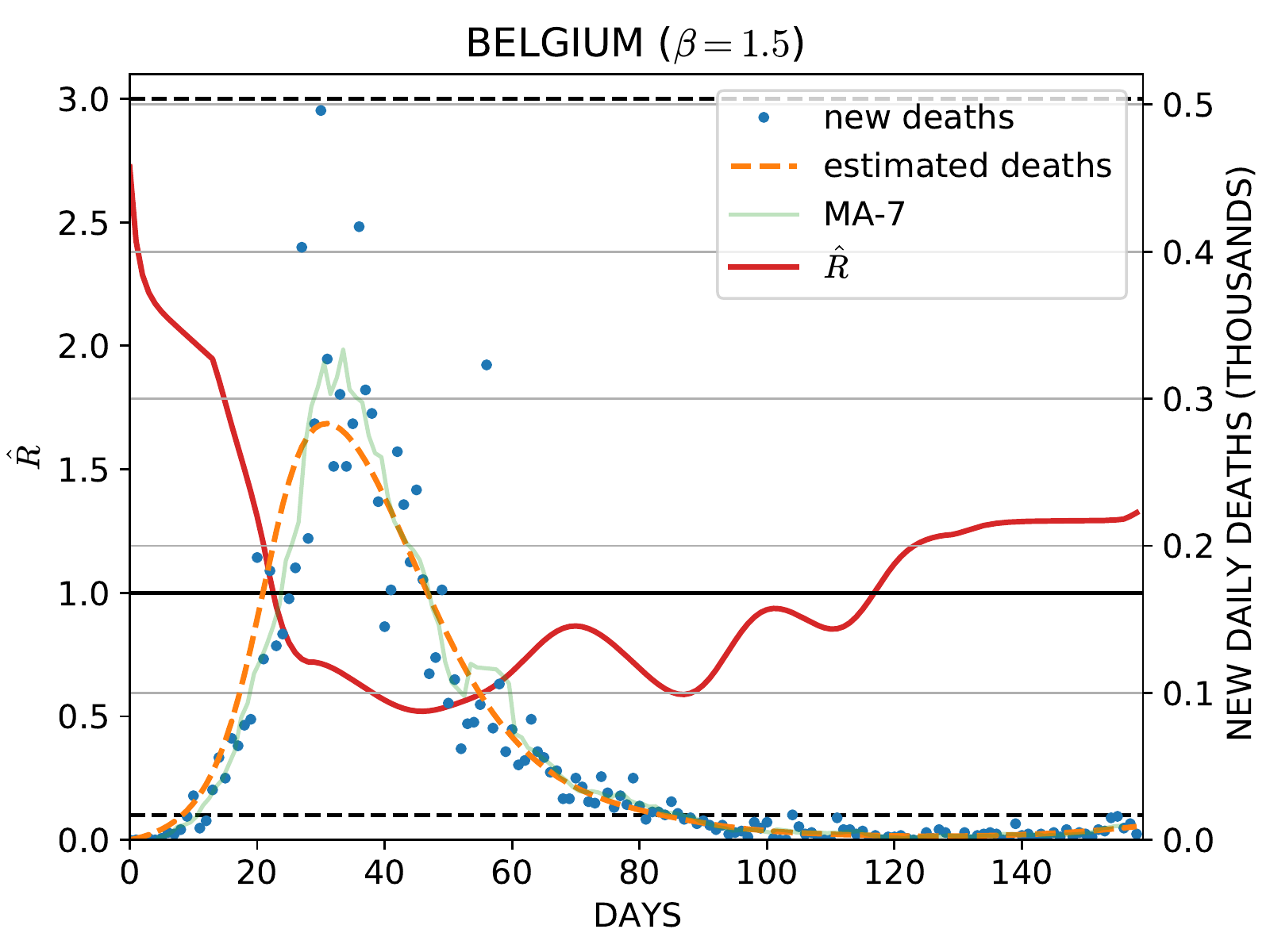}

  \caption{\label{fig:10a}New daily deaths attributed to COVID--19 in
    the \textbf{United States} and \textbf{Belgium}
    from 01/22/2020 through
    08/16/2020~\cite{Dong2020} along with estimates for new
    deaths and $R$ produced by solving Problem~\eqref{eq:15}. The
    value of $\beta$ used is shown in the legends. Also show for
    comparison is a $7$-day moving average of new deaths.
  }
\end{figure}

\begin{figure}
  \includegraphics[width=\columnwidth]{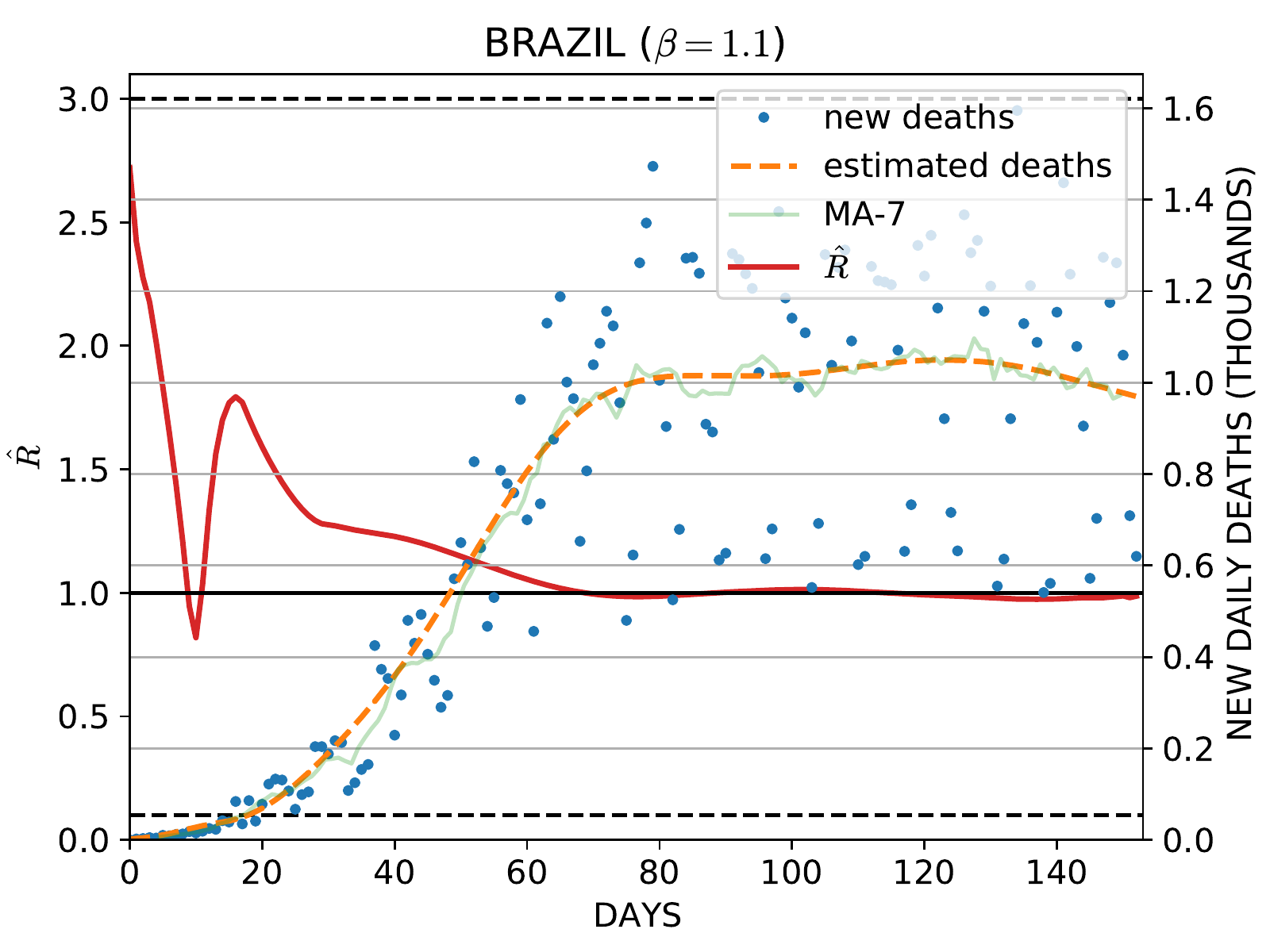}
  \includegraphics[width=\columnwidth]{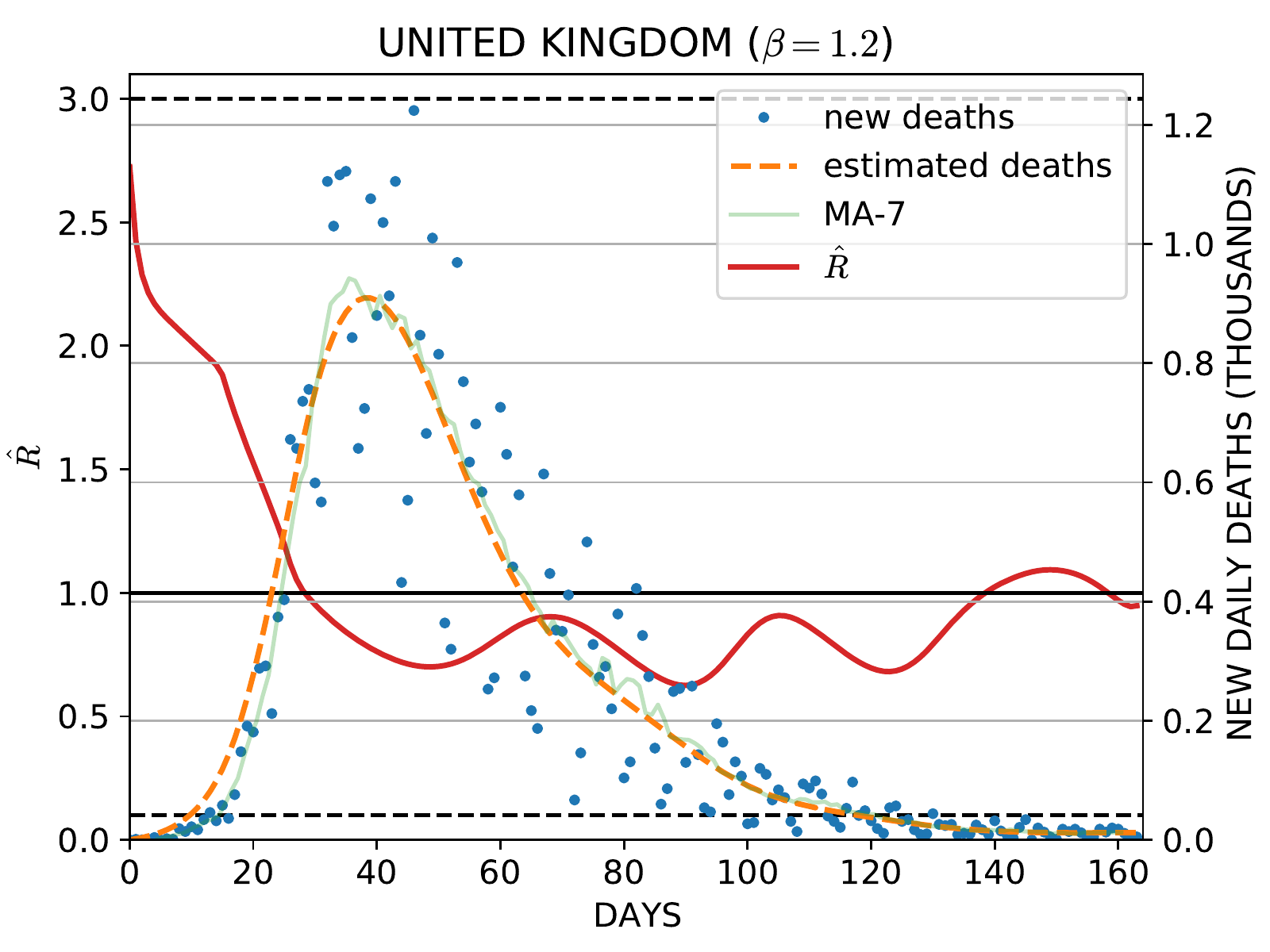}

  \caption{\label{fig:10b}New daily deaths attributed to COVID--19 in
    the \textbf{Brazil} and the \textbf{United Kingdom}
    from 01/22/2020 through
    08/16/2020~\cite{Dong2020} along with estimates for new
    deaths and $R$ produced by solving Problem~\eqref{eq:15}. The
    value of $\beta$ used is shown in the legends. Also show for
    comparison is a $7$-day moving average of new deaths.
  }
\end{figure}

\begin{figure}
  \includegraphics[width=\columnwidth]{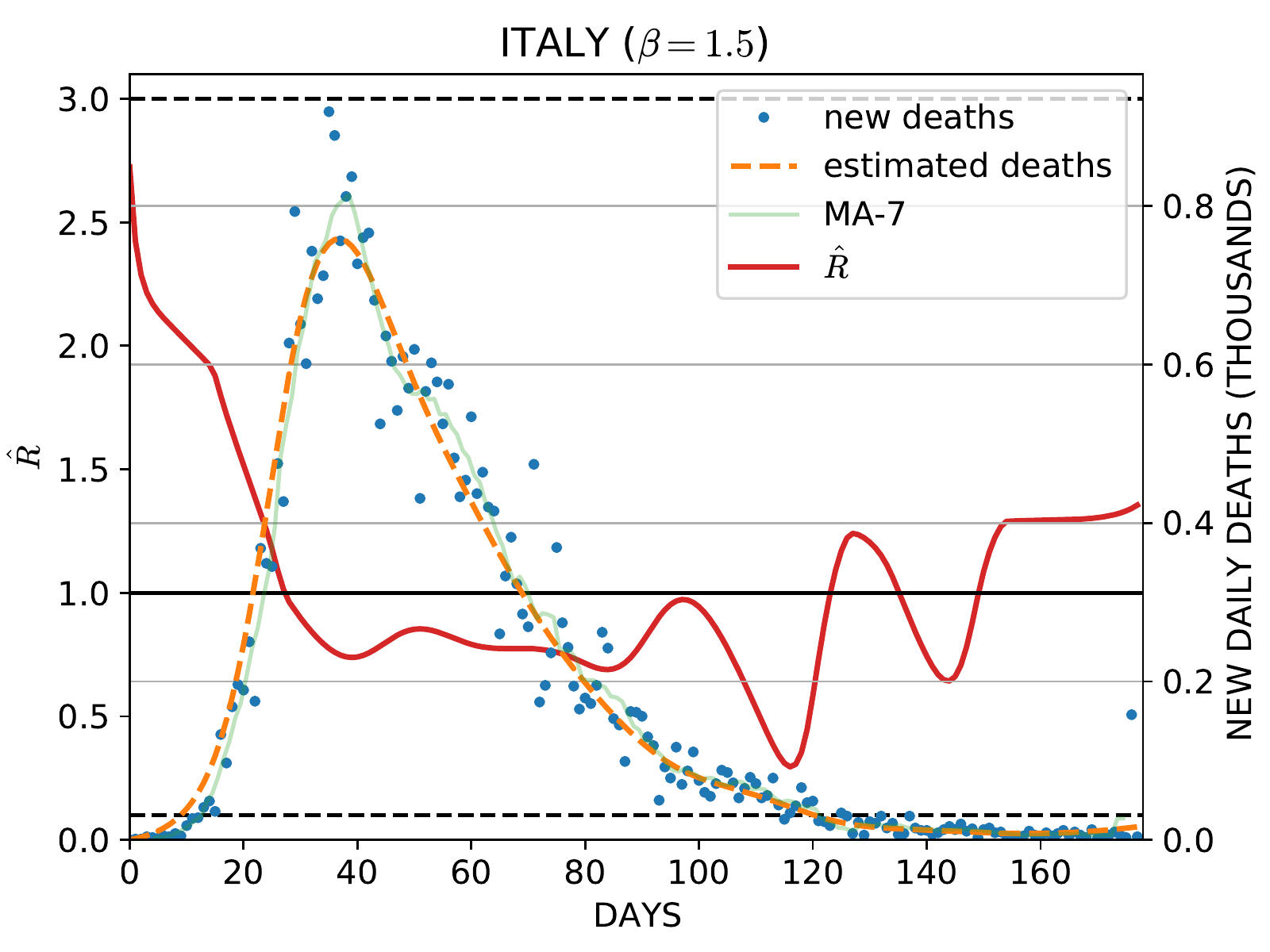}
  \includegraphics[width=\columnwidth]{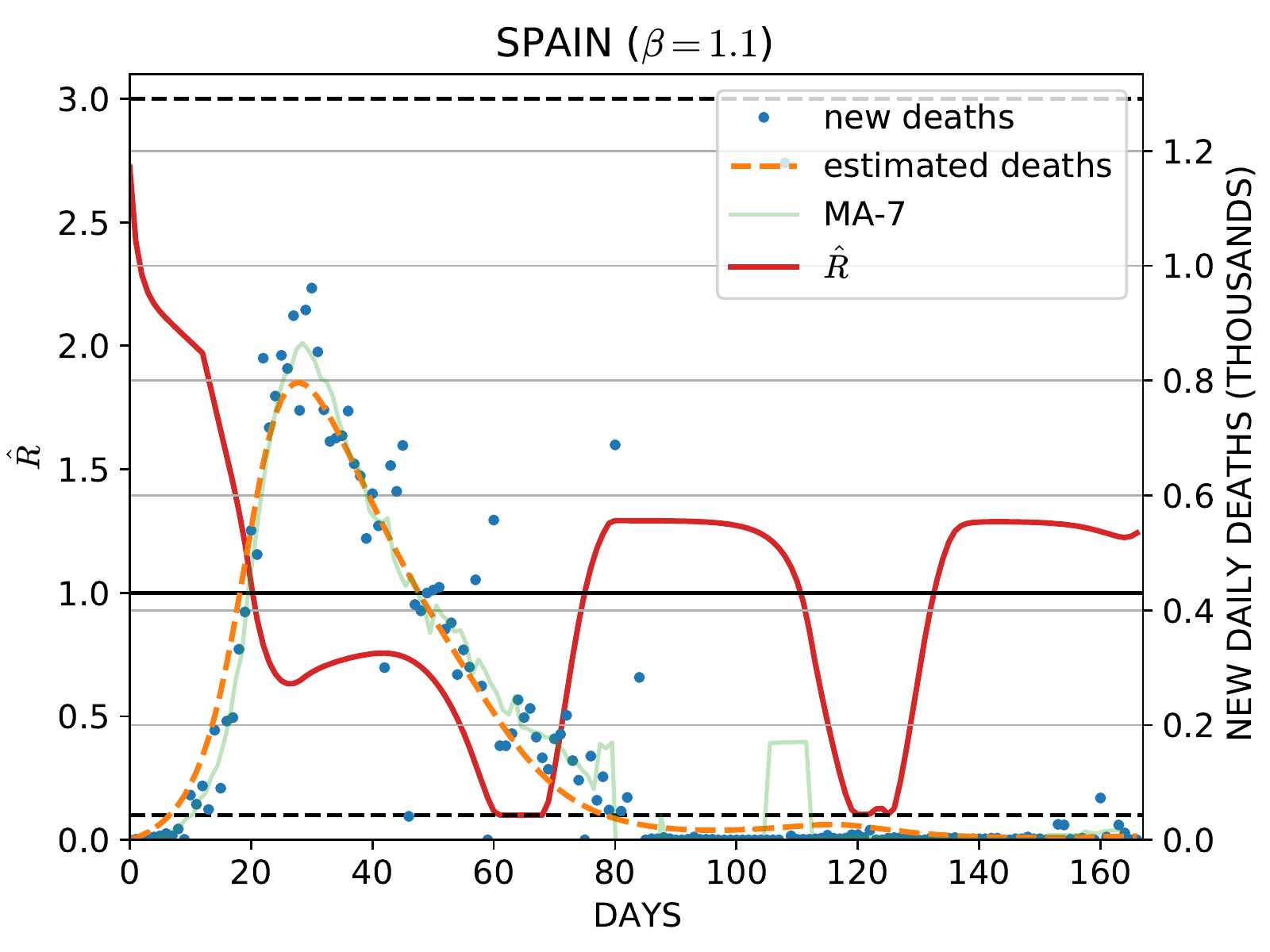}

  \caption{\label{fig:10c}New daily deaths and $7$-day moving average
    attributed to COVID--19 in \textbf{Italy} and \textbf{Spain}
    from 01/22/2020 through 08/16/2020~\cite{Dong2020} along with
    estimates for new deaths and $R$ produced by solving
    Problem~\eqref{eq:15}. The value of $\beta$ used is shown in the
    legends. Also show for comparison is a $7$-day moving average of
    new deaths.}
\end{figure}

\begin{figure}
  \includegraphics[width=\columnwidth]{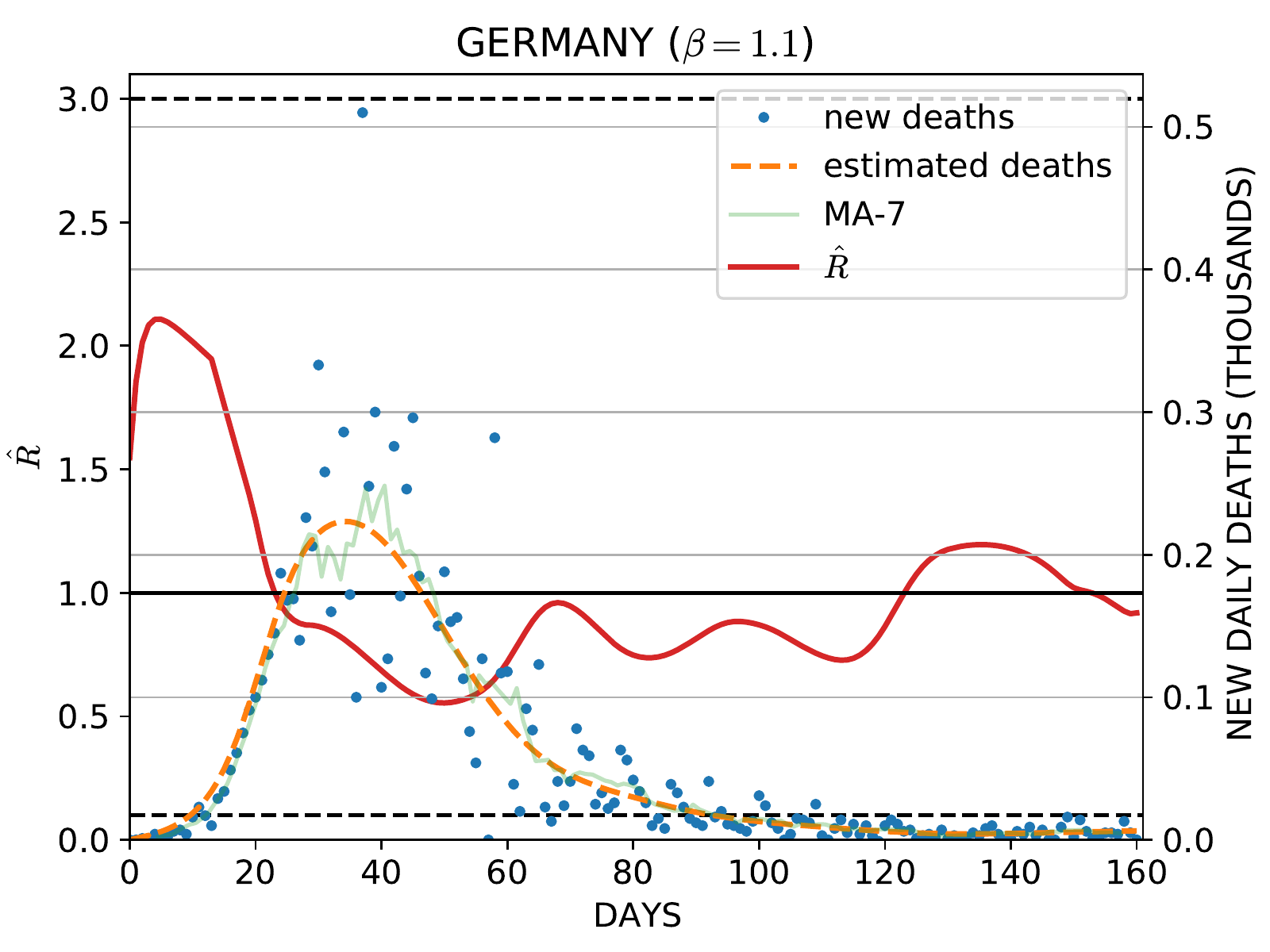}
  \includegraphics[width=\columnwidth]{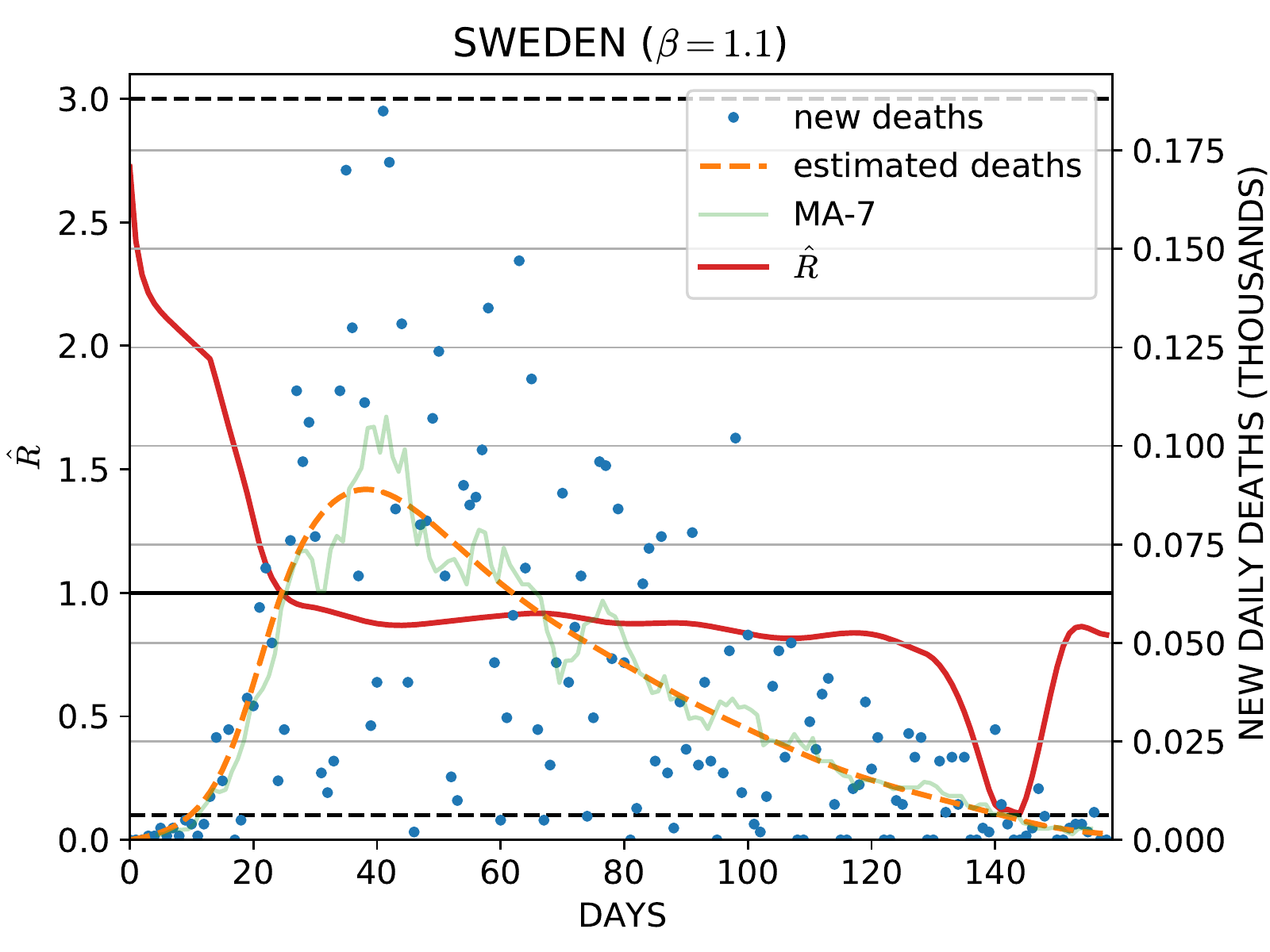}

  \caption{\label{fig:10f}New daily deaths and $7$-day moving average
    attributed to COVID--19 in \textbf{Germany} and \textbf{Sweden}
    from 01/22/2020 through 08/16/2020~\cite{Dong2020} along with
    estimates for new deaths and $R$ produced by solving
    Problem~\eqref{eq:15}. The value of $\beta$ used is shown in the
    legends. Also show for comparison is a $7$-day moving average of
    new deaths.}
\end{figure}

The present paper has revisited unconstrained and proposed new
constrained methods for estimation of the variables and time-varying
parameters of compartmental models with application to the present
COVID-19 epidemic. Even though the underlying model is nonlinear, a
change of coordinates enables the estimation to be done using an
auxiliary linear model that results in convex quadratic optimization
problems which can be solved globally and very efficiently, as well as
be applied to large data sets. The constrained method has been shown
to preserve the physical properties of the model variables throughout
the estimation. Through an additional penalty on the total variation
of the estimates one can trade-off accuracy versus smoothness of the
obtained estimates.

The paper is concluded by showing the results obtained by solving the
constrained estimation Problem~\eqref{eq:15} for COVID-19 death
records for select countries obtained from the COVID-19 Data
Repository by the Center for Systems Science and Engineering (CSSE) at
Johns Hopkins University~\cite{Dong2020}. The selected countries were:
United States of America (already considered earlier), Belgium,
Brazil, United Kingdom, Italy, Spain, Germany, Sweden, all with the
same settings used before in Section~\ref{sec:tradeoff} except for the
parameter $\beta$, which was selected differently for each country
depending on the noise levels of the data. The results and the
corresponding value of $\beta$ is shown for each country in
Figs.~\ref{fig:10a}--\ref{fig:10f}.

\bibliographystyle{ieeetr}
\bibliography{references,books}

\appendix

\label{appendix:direct}

One can rearrange the last two equations
from~\eqref{eq:04a}--\eqref{eq:04c} as in
\begin{align*}
  z_1(k) &= \gamma^{-1} z_2(k+1) + (1 - \gamma^{-1}) z_2(k) \\
  z_2(k) &= \theta^{-1} z_3(k+1) + ( 1 - \theta^{-1}) z_3(k) 
\end{align*}
where $z_1$ and $z_2$ are the system's state and
\begin{align*}
  z_3(k) &= 1 - \delta^{-1} \, y(k)
\end{align*}
can be thought of as an input. The above dynamic system is clearly
non-causal as the present values of the state depend on future values
of the inputs. Making use of the $z$-transform
operator~\cite{Kwakernaak1972} one can relate the transform of the
output $z_1$, $z_2$ with the transform of the input $z_3$ as the
following inproper transfer-functions
\begin{align*}
  Z_1(z) &= \left ( \gamma^{-1} z + (1 - \gamma^{-1}) \right ) \left (
  \theta^{-1} z + ( 1 - \theta^{-1}) \right ) Z_3(z) \\
  Z_2(z) &= \left ( \theta^{-1} z + ( 1 - \theta^{-1}) \right ) Z_3(k) 
\end{align*}
A causal filter can be constructed by delaying the output by two
samples, that is the filter
\begin{align*}
  \hat{Z}_1(z) &= z^{-2} Z_1(z), &
  \hat{Z}_2(z) &= z^{-2} Z_2(k) 
\end{align*}
which corresponds to the state-space
realization~\eqref{eq:008a}--\eqref{eq:008d}. Because matrix $C$ is
invertible, the initial conditions must satisfy
\begin{align*}
  \begin{pmatrix}
    \hat{z}_1(0) \\
    \hat{z}_2(0)
  \end{pmatrix}
  &=
  C
  \begin{pmatrix}
    x_1(0) \\
    x_2(0)
  \end{pmatrix}
  +
  D \, u(0)
\end{align*}
which can be inverted to produce~\eqref{eq:008f}.

\end{document}